\documentclass[reqno]{amsart}
\usepackage[nospace]{cite}
\usepackage{amsmath,amssymb,amsfonts}
\usepackage{epsfig}
\usepackage{amssymb}
\usepackage{dcolumn}
\usepackage{graphicx}
\usepackage{dcolumn}
\usepackage{color}
\usepackage{hyperref}

\newcommand{\real}{\mathbb{R}}
\newcommand{\complex}{\mathbb{C}}

\newcommand{\ep}{\epsilon}
\newcommand{\e}{\mathrm{e}}
\newcommand{\ffi}{\varphi}

\newcommand{\diff}{\,\mathrm{d}}
\newcommand{\dif}{\mathrm{d}}

\newcommand{\sm}{\setminus}

\renewcommand{\le}{\leqslant}
\renewcommand{\ge}{\geqslant}

\DeclareMathOperator \im{Im}
\DeclareMathOperator \re{Re}
\DeclareMathOperator \rge{rge}
\DeclareMathOperator \codim{codim}
\DeclareMathOperator \dist{dist}
\DeclareMathOperator \vect{span}

\newtheorem{theorem}{Theorem}
\newtheorem{lemma}{Lemma}
\newtheorem{proposition}{Proposition}

\newtheorem{remark}{Remark}

\newtheorem{definition}{Definition}
\newenvironment{proofof}
{\noindent{\it Proof of}}{\nolinebreak\hfill$\Box$\bigskip}
\numberwithin{equation}{section}

\begin{document}
\title[Cubic-quintic NLS solitons]{Stable NLS solitons in a cubic-quintic
medium \\
with a delta-function potential}
\author{Fran\c cois Genoud}
\address{Faculty of Mathematics\\
University of Vienna\\
Oskar-Morgenstern-Platz 1\\
1090 Vienna, Austria}
\curraddr{Delft Institute of Applied Mathematics\\
Delft University of Technology\\
Mekelweg 4 \\
2628 CD Delft, The Netherlands}
\email{S.F.Genoud@tudelft.nl}
\author{Boris A. Malomed}
\address{Department of Physical Electronics, School of Electrical Engineering%
\\
Faculty of Engineering\\
Tel Aviv University\\
Tel Aviv 69978, Israel}
\email{malomed@post.tau.ac.il}
\author{Rada M. Weish\"aupl}
\address{Faculty of Mathematics\\
University of Vienna\\
Oskar-Morgenstern-Platz 1\\
1090 Vienna, Austria}
\email{rada.weishaeupl@univie.ac.at}
\subjclass[2000]{35J60; 35B32; 35Q55; 37C75; 74J30; 78A60}
\thanks{We are grateful to Reika Fukuizumi, Katharina Kienecker and Masaya Maeda for
helpful discussions and comments. 
We also thank the anonymous referee for constructive remarks on the manuscript. 
F.G. acknowledges the support of the ERC Advanced
Grant ``Nonlinear studies of water flows with vorticity''. B.A.M.
acknowledges a partial support from the National Science Center of Poland in
the framework of HARMONIA program no.~2012/06/M/ST2/00479. The work of
R.M.W. has been supported by the Hertha-Firnberg Program of the FWF, Grant
T402-N13.}
\keywords{nonlinear Schr\"{o}dinger equation; cubic-quintic nonlinearity;
trapping delta potential; bifurcation; stability}

\begin{abstract}
We study the one-dimensional nonlinear Schr\"{o}dinger equation with the
cubic-quintic combination of attractive and repulsive nonlinearities, and a
trapping potential represented by a delta-function. We determine all bound
states with a positive soliton profile through explicit formulas and, using
bifurcation theory, we describe their behavior with respect to the
propagation constant. This information is used to prove their stability by
means of the rigorous theory of orbital stability of Hamiltonian systems.
The presence of the trapping potential gives rise to a regime where two
stable bound states coexist, with different powers and same propagation
constant.
\end{abstract}

\maketitle


\section{Introduction}

In this paper we study the one-dimensional nonlinear Schr\"{o}dinger (NLS)
equation with the cubic-quintic (CQ) combination of attractive and repulsive
nonlinearities, and a trapping potential represented by a delta-function:
\begin{equation}
i\psi _{z}=-\psi _{xx}-\epsilon \delta (x)\psi -2|\psi |^{2}\psi +|\psi
|^{4}\psi , \quad (x,z)\in\real^2,  \label{nls}
\end{equation}%
for complex $\psi =\psi (x,z)$, and $\epsilon >0$. The objective of the
analysis is the existence and stability of localized \emph{bound states}, in
the form of $\psi (x,z)=\mathrm{e}^{ikz}u(x)$, with $u(x)>0$ satisfying the
respective stationary equation:
\begin{equation}
u^{\prime \prime }-ku+\epsilon \delta (x)u+2u^{3}-u^{5}=0, \quad x\in\real.  \label{solequ}
\end{equation}%
Here and henceforth, $^{\prime }$ stands for differentiation with respect to
$x$. Denoting by $\langle \cdot ,\cdot \rangle $ the duality product between
$H^{-1}(\mathbb{R})$ and $H^{1}(\mathbb{R})$ --- and recalling that $H^{1}(%
\mathbb{R})\subset C(\mathbb{R})\cap L^{\infty }(\mathbb{R})$ ---, the
potential $\delta $ appearing in the soliton equation \eqref{solequ} is the
Dirac distribution at $x=0$, defined by $\langle \delta ,u\rangle =u(0)$ for
all $u\in H^{1}(\mathbb{R})$. In the context of \eqref{nls}, $\delta $ is
interpreted similarly (at each fixed $z$). Hence, \eqref{nls} and %
\eqref{solequ} should be understood in the sense of distributions, even
though the solutions will be smooth outside of $x=0$.

For the sake of brevity, in what follows below we call these
localized bound states `solitons'. Originally, only solitary
waves in integrable systems were called solitons, but in current literature
this term is used in a loose sense, meaning all kinds of stable self-trapped
modes, including those in non-integrable systems.

Problem \eqref{nls}--\eqref{solequ} belongs to a family of models featuring
the competition between self-focusing cubic and defocusing quintic terms,
that have drawn considerable attention in both the physical and the
mathematical communities in recent years, see \cite{bm,jf,gdm,lc,pkh,y,y2}
and the references therein. This combination of nonlinearities is well known
in optical media, including liquid waveguides \cite{Brazil-soliton} and
speciality glasses \cite{Angers}. Especially interesting are colloids
containing metallic nanoparticles, where the CQ nonlinearity can be widely
adjusted by selecting the radius of the suspended nanoparticles and the
colloidal filling factor \cite{Brazil}. Remarkably, the one-dimensional NLS
equation with the CQ nonlinearity admits completely stable exact soliton
solutions \cite{Canada,Pushkarov}, although this equation is not integrable.
The exact soliton solutions are available also in the case
when both the cubic and quintic terms in the one-dimensional NLS equation
have the self-focusing sign \cite{Pelin}. In the absence of linear potential,
the rigorous stability analysis of
one-dimensional NLS solitons with general double-power nonlinearities can be
found in \cite{maeda,ohta}. Then, the effective linear potential term added
to the NLS equation represents a trapping (waveguiding) structure for light
beams, induced by an inhomogeneity of the local refractive index. In
particular, the delta-function term adequately represents a narrow trap
which is able to capture broad solitonic beams.

Existence and stability of bound states of one-dimensional NLS equations
with a delta potential and a single power-law nonlinearity $|\psi
|^{p-1}\psi ,\ p>1$, have been extensively discussed earlier. We refer the
reader to \cite{fo,jf,lc,ma} for more information about this. From the
mathematical point of view, the presence of the delta-function potential has
several interesting consequences. The range of values of the propagation
constant for solutions in free space (i.e., with $\epsilon =0$) is $k\in (0,%
\frac{3}{4})$, and in this case the bifurcation
diagram for the bound states is very simple; see Fig.~\ref{curve0.fig} in
Section~\ref{num.sec}. Namely, the solutions can be parametrized by $k\in (0,%
\frac{3}{4})$, they bifurcate from $u=0$ at $k=0$, their $L^{2}$ norm
(i.e., the integral power of the beam in the optical models) is strictly increasing, and 
diverges as $k\nearrow \frac{3}{4}$. This accounts for the saturation of the nonlinear
refractive index for high-power beams in the CQ optical media. The presence
of the potential gives rise to a fold bifurcation point located to the right
of $k=\frac{3}{4}$. The bifurcating curve now starts off the trivial line at
$k=\frac{\epsilon^{2}}{4}$, can be parametrized by $k$ up to $\overline{k}%
_\epsilon=\frac{3}{4}+\frac{\epsilon ^{2}}{4}$ where it `turns backwards',
and again blows up in $L^{2}(\mathbb{R})$, but now as $k\searrow \frac{3}{4}$. 
The respective bifurcation diagrams are displayed in Fig.~\ref{curve1.fig}%
--\ref{curve3.fig} in Section~\ref{num.sec}, for various values of the
coupling constant $\epsilon >0$. This phenomenon was already observed in
\cite{gdm}, where solitons in a cubic focusing--quintic defocusing medium
with a square-well potential were studied by means of numerical methods and
the variational approximation. In the case of the delta potential considered
here, the fold bifurcation can be described by an exact analysis, as
demonstrated in Section~\ref{bif.sec}. Since the parametrization by $k$
breaks down at $k=\overline{k}_\epsilon$, where the linearization of %
\eqref{solequ} becomes singular, we resort to a result of Crandall and
Rabinowitz \cite{cr73} which provides a natural framework to deal with this
situation.

An important remark at this stage is the multiplicity of positive solutions
of \eqref{solequ} for $k\in (\frac{3}{4},\overline{k}_\epsilon)$. In fact,
the first step of our analysis, in Section~\ref{expsol.sec}, is the \emph{%
explicit determination of all positive solutions} of \eqref{solequ}, in
terms of elementary functions; this is a noteworthy feature of the present
model. Of course, the expressions obtained are somewhat cumbersome, yet we
are able to extract important information from them, notably as regards the
stability of the bound states of \eqref{nls}. We will thus show explicitly
that, for each fixed $k\in (\frac{3}{4},\overline{k}_\epsilon)$, there are
exactly two positive solutions of \eqref{solequ}, and that the corresponding
bound states of \eqref{nls} are both stable. This \emph{bistability}
phenomenon was previously observed numerically in \cite{gdm} for the square-well 
potential; see also \cite{y2}. In the present context, we can prove the
stability rigorously. The fact that the `upper branch' is stable, while the $%
L^{2}$ norm of the solutions is decreasing along it, appeared puzzling when
it was first discovered in \cite{gdm}. However, in the case of a
delta potential considered here, a careful analysis reveals that the spectrum of the
linearization of \eqref{solequ} is strictly positive along
the upper branch, and its stability then follows from the general theory of
orbital stability in \cite{gss}. Along the `lower branch', the linearized
operator has one simple negative eigenvalue, and the rest of its spectrum is
positive. In this case, the Vakhitov--Kolokolov (VK) stability criterion
\cite{vk} (which requires the $L^{2}$ norm to be increasing in $k$)
ensures stability. Note that, for each fixed $k\in (\frac{\epsilon^{2}}{4%
},\frac{3}{4})$, the positive solution of \eqref{solequ} is unique, and the
corresponding bound state is also stable. Therefore, all positive solutions
of \eqref{solequ} give rise to stable bound states of \eqref{nls}. The
stability analysis is carried out in full detail in Section~\ref{stab.sec}.

It is noteworthy that there is no stability swap at the fold bifurcation point, in sharp
contrast with the usual picture in finite-dimensional dynamical systems.
Moreover, the bistability of coexisting bound states with different powers and same
propagation constant offers potential applications to optics in terms of
switching and other elements of all-optical data processing \cite{gdm}.

We would also like to comment on the important role symbolic computer
calculations (using Mathematica) and numerical simulations played in our
analysis. Mathematica was a powerful tool to compute exact formulas that
were too involved to be dealt with manually. This transpires both in the
calculation of solutions in the regime $k\in (\frac{3}{4},\overline{k}%
_\epsilon)$ in Section~\ref{expsol.sec}, and in the stability analysis of
Section~\ref{stab.sec}. On the other hand, numerical experiments were very
useful at early stages of this work, in order to understand the behavior of
solutions, before their explicit representations had been found. We used the
so-called `continuous normalized gradient flow' (CNGF), which was studied
and implemented in \cite{bao} in the context of the NLS equation with a
cubic nonlinearity. The excellent agreement between the numerical and the
exact solutions (see Fig.~\ref{fig2}--\ref{fig6}) demonstrates the
effectiveness of this scheme in the context of \eqref{solequ}. The CNGF
method being based on constraint minimization (see Section~\ref{num.sec}),
this also suggests that the positive solutions of \eqref{solequ} should
admit a variational characterization. Our analytical approach allows us to
describe the spectral and stability properties of the bound states of %
\eqref{nls} without resorting to such a characterization. This
would however present an interest on its own; see for instance \cite{jf} for
results in this direction in the case of a delta potential combined with a
single power nonlinearity.


Lastly, it is relevant to mention that recent numerical and analytical
considerations have demonstrated that the same delta-like attractive
potential may effectively stabilize trapped solitons in the NLS equation
with a combination of defocusing cubic and focusing quintic terms (the signs
opposite to those dealt with in the present work) \cite{Zegadlo}. In the
free-space version of the latter equation, all solitons are completely
unstable.

Note that, in the present work, we have decided to focus on the case of an attractive delta potential, 
i.e., $\ep>0$. Even though this would deserve a rigorous proof, the physical intuition is that
the repulsive case ($\ep<0$) yields unstable solitons, which will tend to escape from the trapped state,
to the left or to the right of $x=0$. We conjecture that, in this case, the 
$H^1$ solutions are unstable under general $H^1$ perturbations, 
but remain stable under radial perturbations, as in the case of a single-power nonlinearity studied
in \cite{jf,lc}.


\section{Explicit solutions}

\label{expsol.sec}

We first establish some elementary properties of $H^1$ solutions of \eqref{solequ}. 
In particular they are all positive. Sign-changing
solutions exist in the form of cnoidal waves pinned to the delta potential, but those
are periodic solutions, not localized ones.

\begin{proposition}
\label{basic.prop} Let $k>0$ and $u\in H^1(\mathbb{R})$ be a 
non-trivial solution of \eqref{solequ}. Then $u$ satisfies:

\begin{itemize}
\item[(i)] $u^{\prime\prime}- ku + 2u^3 - u^5 = 0, \ x\neq0$;

\item[(ii)] $\pm u>0$ on $\real$;

\item[(iii)] $u$ is even on $\mathbb{R}$;

\item[(iv)] $u\in C^2(\mathbb{R}\setminus\{0\})\cap C(\mathbb{R})$;

\item[(v)] $u^{\prime}(0^\pm)=\mp(\epsilon/2)u(0)$;

\item[(vi)] $u(x), u^{\prime}(x) \to 0$ as $|x|\to\infty$.
\end{itemize}
\end{proposition}

\begin{proof}
Properties (i) and (iv) follow by a standard bootstrap argument 
using test functions in $C^\infty_0(\real\sm\{0\})$ (see for instance \cite[Section~8]{cazenave}). 
This argument in fact yields 
$u\in H^2(\real\sm\{0\})$, from which (vi) follows.

For (ii), first observe that, if $u$ is a solution, so is $-u$. Now suppose by contradiction that
there exists $x_0>0$ such that $u(x_0)=0$ (the case $x_0<0$ is handled similarly). If 
$u'(x_0)=0$, Cauchy's uniqueness theorem implies $u\equiv 0$ on $(0,\infty)$. Since $u$ is
continuous, \eqref{jump} below then implies $u\equiv 0$ on $\real$. Suppose now that 
$u'(x_0)\neq0$. Multiplying the equation by $u'$ and integrating from $x_0$ to $x>0$ yields
$$
(u')^2(x)-ku^2(x)+u^4(x)-\frac13u^6(x)=(u')^2(x_0)>0 \quad\text{for all} \ x>0.
$$
However, integrating from $x$ to $+\infty$ and using $u\in H^2(\real\sm\{0\})$ yields
$$
(u')^2(x)-ku^2(x)+u^4(x)-\frac13u^6(x)=0 \quad\text{for all} \ x>0.
$$
This contradiction shows that $\pm u>0$ on $\real$.

To prove (v), one first establishes that 
\begin{equation}\label{jump}
u^{\prime}(0^+)-u^{\prime}(0^-)=-\epsilon u(0)
\end{equation}
by integrating \eqref{solequ} over $[-t,t]$ and letting $t\to0^+$. Then, multiplying
\eqref{solequ} by $u'$ and integrating from $0^+$ to $+\infty$ yields
$$
(u')^2(0^+)=ku^2(0)-u^4(0)+\frac13 u^6(0).
$$
Similarly, integrating from $-\infty$ to $0^-$ yields
$$
(u')^2(0^-)=ku^2(0)-u^4(0)+\frac13 u^6(0),
$$
so that $|u'(0^+)|=|u'(0^-)|$. Now, if $u'(0^+)=u'(0^-)$ then $u(0)=0$ by \eqref{jump}.
If $u'(0^+)=u'(0^-)=0$ then Cauchy's theorem implies $u\equiv 0$. On the other hand,
if $u'(0^+)=u'(0^-)\neq0$ then $u$ becomes negative close to $x=0$, a contradiction.
Therefore, $u'(0^+)=-u'(0^-)$ and (iv) follows from \eqref{jump}.

 Finally, (iii) follows by observing that $w(x):=u(x)-u(-x)$
 satisfies the initial value problem
 \begin{equation*}
 w^{\prime \prime }(x)-kw(x)+2a(x)w(x)-b(x)w(x)=0, \quad
 w(0)=w^{\prime}(0)=0,
 \end{equation*}
 where
 \begin{equation*}
 a(x)=u^2(x)+u(x)u(-x)+u^2(-x), \quad
 b(x)=u^4(x)+u^3(x)u(-x)+u^2(x)u^2(-x)+u(x)u^3(-x)+u^4(-x).
 \end{equation*}
Cauchy's theorem then implies $w\equiv0$.
\end{proof}

We shall henceforth focus on positive solutions. We will show that
they can all be expressed in terms of elementary functions, which is a remarkable
feature of the present model. This is
especially striking in the range of the propagation constant $k>3/4$ which
is not allowed in free space (i.e., when $\epsilon=0$) \cite%
{Canada,Pushkarov}. 

First, multiplying \eqref{solequ} by $u^{\prime }$ and integrating from $x>0$
to $\infty $, respectively from $-\infty $ to $x<0$, we get
\begin{equation}
(u^{\prime})^2(x)-ku^{2}(x)+u^{4}(x)-(1/3)u^{6}(x)=0,\quad x\neq 0.
\label{1stint}
\end{equation}
In particular, taking the limit $x\rightarrow 0^{\pm }$ and using
Proposition~\ref{basic.prop}~(v) yields, assuming that $u(0)\neq 0$,
\begin{equation}
u^{4}(0)-3u^{2}(0)+3(k-\epsilon ^{2}/4)=0,
\end{equation}
the solutions of which are
\begin{equation}  \label{uopm}
u_{\pm ,k,\epsilon }^{2}(0)= \textstyle\frac{3}{2}\left(1\pm \sqrt{1-\frac{4%
}{3}\left( k-\frac{\epsilon ^{2}}{4}\right) }\right).
\end{equation}
Note that both $u^2_{\pm ,k,\epsilon}(0)$ exist and are positive if and only
if
\begin{equation}
\frac{\epsilon ^{2}}{4}<k\leqslant \frac{3}{4}+\frac{\epsilon ^{2}}{4}.
\label{existcond}
\end{equation}
Next, with a view of further integrating \eqref{1stint}, we express $%
u^{\prime}$ as
\begin{equation}
u^{\prime }(x)=\pm u(x)\textstyle\sqrt{\frac{1}{3}u^{4}(x)-u^{2}(x)+k},\quad
x\neq 0  \label{u'}
\end{equation}
(recall we seek solutions with $u>0$). The positivity condition for %
\eqref{u'} to hold reads $u^{2}(x)\in (0,\tilde{u}_{-}^{2})\cup (\tilde{u}%
_{+}^{2},\infty )$, where
\begin{equation}
\tilde{u}_{\pm }^{2}=\textstyle\frac{3}{2}\left( 1\pm \sqrt{1-\frac{4k}{3}}%
\right) .
\end{equation}
Since positive solutions satisfying \eqref{u'} are even and strictly
decreasing in $x>0$, the continuity and the decay of $u$ at infinity only allow
for\footnote{%
Note that $\tilde{u}_{-}^{2}=\psi _{\text{sol}}(0,0)^{2}$ in \cite{bm}.}
\begin{equation}
u^{2}(0)\leqslant \tilde{u}_{-}^{2}= \textstyle\frac{3}{2}\left( 1-\sqrt{1-%
\frac{4k}{3}}\right) ,\quad \frac{\epsilon ^{2}}{4}<k<\frac{3}{4}.
\label{poscond}
\end{equation}
If $k> 3/4$, \eqref{u'} is well defined without further restriction on $u(0)$%
, and condition \eqref{poscond} is void. (The nature of the degeneracy at $%
k=3/4$ will become more apparent later.) In view of \eqref{uopm}, %
\eqref{existcond} and \eqref{poscond}, we identify two different regimes
(see the bifurcation diagrams for various values of $\epsilon$ in Section~%
\ref{num.sec}):

\begin{itemize}
\item[(A)] $\frac{\epsilon ^{2}}{4}<k<\frac{3}{4}$: there is only one
soliton, $u_{-,k,\epsilon}$, corresponding to $u_{-,k,\epsilon}^{2}(0)$;

\item[(B)] $\frac{3}{4} < k < \frac{3}{4} +\frac{\epsilon ^{2}}{4}$: there
are two different solitons, $u_{\pm,k,\epsilon}$, corresponding respectively
to $u_{\pm,k,\epsilon}^{2}(0)$.
\end{itemize}

\medskip Notice, in particular, that regime (A) is void if $\epsilon
\geqslant \sqrt{3}$; so \emph{we will suppose $0<\epsilon <\sqrt{3}$ from
now on}. Also, we already see from the above analysis that a fold
bifurcation occurs at $\overline{k}_\epsilon=\frac{3}{4}+\frac{\epsilon ^{2}%
}{4}$ where two distinct solutions merge and disappear (there is no soliton
for $k>\overline{k}_\epsilon$). 

From \eqref{u'} and the previous discussion, any positive solution of %
\eqref{solequ} with $\frac{\epsilon ^{2}}{4}<k<\frac{3}{4}+\frac{\epsilon
^{2}}{4}$ decaying at infinity satisfies
\begin{equation}
u^{\prime }(x)=-\,\mathrm{sgn}(x)u(x)\textstyle\sqrt{\frac{1}{3}
u^{4}(x)-u^{2}(x)+k},\quad x\neq 0.  \label{sepvar}
\end{equation}
In particular, $u$ is even, $u^{\prime }(x)<0$ for $x>0$, and $%
\lim_{x\rightarrow \infty }u^{\prime }(x)/u(x)=-\sqrt{k}$, so $u(x)$ decays
like $\mathrm{e}^{-\sqrt{k}|x|}$ as $|x|\to\infty$.

Now \eqref{sepvar} is a first order ODE with separated variables, which can
be integrated explicitly. Alternatively, the solutions in regime (A) are
easily constructed by applying some surgery to the known explicit solitons
in free space, given in \cite{Canada,Pushkarov} as
\begin{equation}
u_{-,k,0}(x)=\sqrt{\frac{2k}{1+\sqrt{1-\frac{4k}{3}}\cosh \big(2\sqrt{k}x%
\big)}},\quad \textstyle0<k<\frac{3}{4}.  \label{freespacesol}
\end{equation}
The corresponding solutions pinned to the delta potential with $\epsilon >0$
are obtained as
\begin{equation}
u_{-,k,\epsilon }(x)=\sqrt{\frac{2k}{1+\sqrt{1-\frac{4k}{3}} \cosh\big(2%
\sqrt{k}(|x|+\xi )\big)}},\quad \textstyle\frac{\epsilon ^{2}}{4}<k<\frac{3}{%
4},  \label{pinnedsol}
\end{equation}
where $\xi =\xi (k,\epsilon )$ is determined by the jump condition in
Proposition~\ref{basic.prop}~(v), which yields
\begin{equation*}
\frac{\sinh (2\sqrt{k}\xi )}{1+\sqrt{1-\frac{4k}{3}} \cosh \big(2\sqrt{k}\xi%
\big)}=\frac{\epsilon }{2\sqrt{k}}\frac{1}{\sqrt{1-\frac{4k}{3}}}.
\end{equation*}
It is not difficult to check that this equation has a unique solution $%
\xi\in \mathbb{R}$ if $k>\frac{\epsilon ^{2}}{4}$. In fact this solution can
be computed explicitly:
\begin{equation*}
\mathrm{e}^{2\sqrt{k}\xi }= \frac{\epsilon +\epsilon \sqrt{1+(\frac{4k}{%
\epsilon ^{2}}-1)(1-\frac{4k}{3})}}{2(\sqrt{k}-\frac{\epsilon }{2}) \sqrt{1-%
\frac{4k}{3}}}.
\end{equation*}
Thus, the solutions in \eqref{pinnedsol} take the form of
\begin{equation}
u_{-,k,\epsilon }(x)=\sqrt{\frac{2k}{1+\frac{\epsilon + \epsilon \sqrt{%
1+(4k/\epsilon ^{2}-1)(1-4k/3)}}{4(\sqrt{k}-\epsilon /2)}\mathrm{e}^{2\sqrt{k%
}|x|} +\frac{(1-4k/3)(\sqrt{k}-\epsilon /2)}{\epsilon +\epsilon \sqrt{%
1+(4k/\epsilon ^{2}-1)(1-4k/3)}}\mathrm{e}^{-2\sqrt{k}|x|}}},\quad \textstyle%
\frac{\epsilon ^{2}}{4}<k<\frac{3}{4}.  \label{expinnedsol}
\end{equation}


For $k=3/4$, a similar procedure applied to the `front soliton' given in
Eq.~(11) of \cite{bm} yields a solution
\begin{equation}  \label{3/4}
u_{f,\epsilon}(x)= \sqrt{\frac{3}{2}}\left[1+\frac{\epsilon }{\sqrt{3}%
-\epsilon }\mathrm{e}^{\sqrt{3}|x|}\right]^{-1/2}.
\end{equation}

As can be seen in Section~\ref{num.sec} by comparing the bifurcation
diagrams for solutions in free space to those with $\epsilon >0$, solutions
with $k>3/4$ only exist in the presence of the potential. In other
words, regime (B) above is void for $\epsilon =0$. Therefore, no free-space
solutions are available that could be pinned to the delta potential by the
same sort of surgery as above, and one has to integrate the equation
manually. We integrate \eqref{sepvar} using an Euler substitution, which
yields
\begin{equation}
u_{\pm ,k,\epsilon }(x)=2\sqrt{\frac{k}{\left( \mathrm{e}^{\sqrt{k}(|x|-c)}+
\mathrm{e}^{-\sqrt{k}(|x|-c)}\right) \Big(\Big(2\sqrt{\frac{k}{3}}+1\Big)
\mathrm{e}^{\sqrt{k}(|x|-c)}-\Big(2\sqrt{\frac{k}{3}}-1\Big)\mathrm{e}^{-%
\sqrt{k}(|x|-c)}\Big)}},\quad \textstyle\frac{3}{4}<k<\frac{3}{4}+\frac{%
\epsilon ^{2}}{4},  \label{regime2}
\end{equation}
where the integration constant $c=c_{\pm ,k,\epsilon }\in \mathbb{R}$ can be
determined from \eqref{uopm}. The expressions for the integration constants
are somewhat cumbersome. They can be computed using Mathematica, which yields%
\footnote{%
The positivity of the expressions under the square roots can be checked by
plotting their graphs (as functions of $\epsilon $ and $k$) in Mathematica.}
\begin{equation*}
\mathrm{e}^{\sqrt{k}c_{-,k,\epsilon }}= \sqrt{\frac{3-\sqrt{3}\sqrt{%
3+\epsilon ^{2}-4k}+2\epsilon \sqrt{k}-4k}{-3+\sqrt{3}\sqrt{3+\epsilon
^{2}-4k}+2\sqrt{3}\sqrt{k}-2\sqrt{k}\sqrt{3+\epsilon ^{2}-4k}}}
\end{equation*}
and
\begin{equation*}
\mathrm{e}^{\sqrt{k}c_{+,k,\epsilon }}= \sqrt{\frac{-3-\sqrt{3}\sqrt{%
3+\epsilon ^{2}-4k} -2\epsilon \sqrt{k}+4k}{3+\sqrt{3}\sqrt{3+\epsilon
^{2}-4k} -2\sqrt{3}\sqrt{k}-2\sqrt{k}\sqrt{3+\epsilon ^{2}-4k}}}.
\end{equation*}

Hence the explicit form \eqref{regime2} is not very convenient to work with,
but we shall see in Section~\ref{stab.sec} that some information can
nevertheless be extracted from it. However, for given values of the
parameters, the exact form of the solutions may be useful, especially in
numerical calculations. For instance, at the fold bifurcation point, where $%
\overline{k}_\epsilon=\frac{3}{4}+\frac{\epsilon ^{2}}{4}$, the solution
takes the more tractable form:
\begin{equation*}
\overline{u}_{\epsilon }(x)=\sqrt{\frac{3}{2}}\sqrt{\frac{3+\epsilon ^{2}}{%
3+\epsilon ^{2} \cosh (\sqrt{3+\epsilon ^{2}}|x|)+\epsilon \sqrt{%
3+\epsilon^{2}}\sinh (\sqrt{3+\epsilon ^{2}}|x|)}}.  \label{ubar}
\end{equation*}

\begin{remark}
\rm 
It can also be checked that, as $\epsilon \rightarrow 0$, the
solutions in \eqref{expinnedsol} converge to the corresponding free-space
solitons in \eqref{freespacesol}. It will be seen in the proof of Lemma~\ref%
{holomorphy.lem} that, in fact, they can be extended to a holomorphic family
of functions parametrized by $\epsilon$ in a complex domain containing zero.
\end{remark}


\section{The bifurcation analysis}

\label{bif.sec}

In this section we will embed the above explicit solutions in a
bifurcation-theoretic framework, suitable to the rigorous stability analysis
which will be carried out in Section~\ref{stab.sec}. We will prove the
following result.

\begin{theorem}
\label{curve.thm} Let $\epsilon\in(0,\sqrt{3})$. The solutions $(k,u)$ of %
\eqref{solequ} obtained in \eqref{expinnedsol}--\eqref{regime2} form a
smooth curve in $\mathbb{R}\times H^{1}(\mathbb{R})$, which bifurcates from
the trivial solution $u\equiv 0$ at $k=\frac{\epsilon ^{2}}{4}$, consists of
the solutions $u_{-,k,\epsilon }$ up to $\overline{k}_\epsilon=\frac{3}{4}+%
\frac{\epsilon ^{2}}{4}$, where it has a turning point, and then consists of
the solutions $u_{+,k,\epsilon }$ and becomes unbounded as $k\searrow 3/4$.
More precisely,
\begin{equation*}
\lim_{k\searrow \frac{\epsilon ^{2}}{4}}\Vert u_{-,k,\epsilon }\Vert
_{H^{1}}=0\quad \text{and}\quad \lim_{k\searrow 3/4}\Vert
u_{+,k,\epsilon}\Vert _{L^{2}}=\infty .
\end{equation*}
\end{theorem}

A good mental picture of Theorem~\ref{curve.thm} can be grasped from the
bifurcation diagrams in Section~\ref{num.sec}, where $\Vert
u_{+,k,\epsilon}\Vert_{L^2}$ is plotted against $k$, for various values of
the coupling constant $\epsilon>0$.

%

To prove Theorem~\ref{curve.thm}, first observe that Eq.~\eqref{solequ} can
be formulated as
\begin{equation}
F_{\epsilon }(k,u)=0,  \label{functequ}
\end{equation}
where
\begin{equation}
F_{\epsilon }(k,u):=u^{\prime \prime } -k u + \epsilon \delta(x) u+2
u^{3}-u^{5}  \label{F}
\end{equation}
can be seen as a mapping $F_{\epsilon }:\mathbb{R}\times H^{1}(\mathbb{R}%
)\rightarrow H^{-1}(\mathbb{R})$ (by interpreting the right-hand side as a
distribution). It is standard to show that this mapping is continuously Fr%
\'{e}chet differentiable. The derivative with respect to $u$, which will
play a key role in our analysis, is formally given by $D_{u}F_{%
\epsilon}(k,u):H^{1}(\mathbb{R})\rightarrow H^{-1}(\mathbb{R})$,
\begin{equation*}  \label{DuF}
D_{u}F_{\epsilon }(k,u)v=v^{\prime \prime }-kv+\epsilon \delta (x)v
+[6-5u^{2}] u^{2}v,\quad v\in H^{1}(\mathbb{R}).
\end{equation*}
More precisely, following the proof of \cite[Lemma~10]{lc}, 
$D_{u}F_{\epsilon }(k,u)$ can be interpreted as a self-adjoint operator
acting in $L^{2}(\mathbb{R})$, with domain
\begin{equation*}
\mathcal{D}_{\epsilon }= \left\{v\in H^{1}(\mathbb{R})\cap H^{2}(\mathbb{R}%
\setminus \{0\}):v^{\prime}(0^+)-v^{\prime}(0^-)=-\epsilon v(0)\right\},
\end{equation*}
defined by
\begin{equation*}
D_{u}F_{\epsilon }(k,u)v=v^{\prime \prime} -k v+[6-5u^2]u^{2}v,\quad v\in
\mathcal{D}_{\epsilon}.
\end{equation*}

Using the explicit formulas for the solutions obtained in Section~\ref%
{expsol.sec} (in particular their uniform exponential decay), it can be
shown that
\begin{equation*}
\mathcal{S}_{-,\epsilon }= \big\{(k,u_{-,k,\epsilon }):k\in (\textstyle\frac{%
\epsilon ^{2}}{4},\overline{k}_\epsilon)\big\}\quad \text{and}\quad \mathcal{%
S}_{+,\epsilon }= \big\{(k,u_{+,k,\epsilon }):k\in (\textstyle\frac{3}{4},%
\overline{k}_\epsilon)\big\}
\end{equation*}
define two continuous curves in $\mathbb{R}\times H^{1}(\mathbb{R})$. In the
remainder of the paper, we will obtain much more information about these
sets. It will be convenient to call $\mathcal{S}_{-,\epsilon }$ the \emph{%
lower curve} and $\mathcal{S}_{+,\epsilon }$ the \emph{upper curve}.

\begin{proposition}
\label{curves.prop} The sets $\mathcal{S}_{\epsilon ,\pm }$ are smooth
curves of non-degenerate solutions of \eqref{functequ}, in the sense that $%
D_{u}F_{\epsilon }(k,u)$ is non-singular along $\mathcal{S}_{-,\epsilon }$
and $\mathcal{S}_{+,\epsilon }$. Furthermore, $\mathcal{S}_{-,\epsilon}$
bifurcates from the point $(\frac{\epsilon^{2}}{4},0)$ in $\mathbb{R}\times
H^{1}(\mathbb{R})$, and meets $\mathcal{S}_{+,\epsilon }$ at the point $(%
\overline{k}_\epsilon,\overline{u}_{\epsilon })$, where $D_{u}F_{\epsilon
}(k,u)$ becomes singular.
\end{proposition}

\begin{proof}
First, it is easily seen that
\begin{equation*}
\ker D_{u}F_{\epsilon }(\textstyle\frac{\epsilon ^{2}}{4},0)= \vect\big\{%
\mathrm{e}^{-\frac{\epsilon}{2}|x|}\big\},
\end{equation*}
so that zero is a simple eigenvalue of $D_{u}F_{\epsilon }(\frac{\epsilon^{2}%
}{4},0)$. It then follows from standard bifurcation theory that $\mathcal{S}%
_{-,\epsilon }$ bifurcates from $(\frac{\epsilon ^{2}}{4},0)$. More
precisely, the Crandall-Rabinowitz theorem \cite[Theorem~1.7]{cr71} yields
the existence of a unique local continuous curve of solutions bifurcating
from the line of trivial solutions $\{(k,0):k\in \mathbb{R}\}$ in $\mathbb{R}%
\times H^{1}(\mathbb{R})$ at the point $(\frac{\epsilon ^{2}}{4},0)$. Since
our explicit solutions all belong to $H^{1}(\mathbb{R})$, they coincide with
the Crandall-Rabinowitz curve in a neighborhood of $(\frac{\epsilon ^{2}}{4}%
,0)$ in $\mathbb{R}\times H^{1}(\mathbb{R})$.

The smoothness of the curves $\mathcal{S}_{-,\epsilon }$ and $\mathcal{S}%
_{+,\epsilon }$ follows from the implicit function theorem in $\mathbb{R}%
\times H^{1}(\mathbb{R})$, provided that $D_{u}F_{\epsilon }(k,u):H^{1}(%
\mathbb{R})\rightarrow H^{-1}(\mathbb{R})$ is non-singular along the
solution curves, which is given by Lemma~\ref{nondegen.lem} below.
\end{proof}

\begin{remark}
\rm 
Note that the global bifurcation theorem of Rabinowitz \cite[%
Theorem~1.3]{r} could be applied here to deduce the existence of a connected
set of solutions, bifurcating from the point $(\frac{\epsilon^{2}}{4},0)$,
where we have only shown that local bifurcation occurs. But our proof is
more straightforward --- we construct the curve using the explicit solutions
we found in Section~\ref{expsol.sec} --- and provides a smooth curve, which
is stronger than the conclusion of Rabinowitz's theorem. We shall use
Rabinowitz's theorem to infer that the bifurcating curve is unbounded (see
the conclusion of the proof of Theorem~\ref{curve.thm} below). But this can
also be deduced from the explicit form of the solutions (see Eq.~%
\eqref{explicitslopes} below). 
\end{remark}

In view of the more detailed spectral analysis that will be carried out
later, and in order to follow the usual sign convention of the spectral
theory of Schr\"{o}dinger operators, it is convenient to introduce the
self-adjoint operators $T_{\pm ,k,\epsilon}: \mathcal{D}_{\epsilon }\subset
L^{2}(\mathbb{R})\rightarrow L^{2}(\mathbb{R})$,
\begin{equation}
T_{\pm ,k,\epsilon }v:=-D_{u}F_{\epsilon }(k,u_{\pm ,k,\epsilon})v
=-v^{\prime \prime }+kv-[6-5u_{\pm ,k,\epsilon }^{2}(x)]u_{\pm,k,\epsilon
}^{2}(x)v.  \label{linear}
\end{equation}
Formally, one can write
$$
T_{\pm ,k,\epsilon }v=
-v^{\prime \prime }+kv-\ep\delta(x)v-[6-5u_{\pm ,k,\epsilon }^{2}(x)]u_{\pm,k,\epsilon}^{2}(x)v.
$$
Thus, $T_{\pm ,k,\epsilon }$ can be seen as an operator acting
between $H^{1}(\mathbb{R})$ and $H^{-1}(\mathbb{R})$, by interpreting the
right-hand side as a distribution.

The proof of the next lemma relies on the spectral theory of 
$-\frac{\mathrm{d}^{2}}{\mathrm{d}x^{2}}+k: \mathcal{D}_{\epsilon }\subset
L^{2}(\mathbb{R})\rightarrow L^{2}(\mathbb{R})$, which has been formalized in the classic
book \cite{alb}.

\begin{lemma}
\label{nondegen.lem} The linearized operator \eqref{DuF} satisfies:

\begin{itemize}
\item[(i)] $T_{-,k,\epsilon}:H^1(\mathbb{R})\to H^{-1}(\mathbb{R})$ is an
isomorphism for all $k\in\textstyle(\frac{\epsilon^2}{4},\overline{k}%
_\epsilon)$;

\item[(ii)] $T_{+,k,\epsilon}:H^1(\mathbb{R})\to H^{-1}(\mathbb{R})$ is an
isomorphism for all $k\in\textstyle(\frac34,\overline{k}_\epsilon)$;

\item[(iii)] $D_uF_\epsilon(\overline{k}_\epsilon,\overline{u}_\epsilon)$ is
singular with
\begin{equation}  \label{kerspan}
\ker D_uF_\epsilon(\overline{k}_\epsilon,\overline{u}_\epsilon)= \vect%
\{\eta_\epsilon\}, \quad \eta_\epsilon=|\overline{u}_\epsilon^{\prime }|.
\end{equation}
Furthermore, since $\eta_\epsilon>0$, zero is the principal eigenvalue of $%
D_uF_\epsilon(\overline{k}_\epsilon,\overline{u}_\epsilon)$.
\end{itemize}
\end{lemma}

\begin{proof}
A first important remark is that each operator $T_{\pm ,k,\epsilon }$ is a
compact perturbation of $-\frac{\mathrm{d}^{2}}{\mathrm{d}x^{2}}+k: H^{1}(%
\mathbb{R})\rightarrow H^{-1}(\mathbb{R})$, the latter being an isomorphism
for all $k>0$. It then follows from standard spectral theory (see \cite{kato,s,alb}) 
that the spectrum of $T_{\pm ,k,\epsilon }:\mathcal{D}_{\epsilon }\subset
L^{2}(\mathbb{R})\rightarrow L^{2}(\mathbb{R})$ consists of a
finite number of isolated simple eigenvalues lying below a
continuous part $[k,\infty )$. Furthermore, $T_{\pm ,k,\epsilon }:H^{1}(%
\mathbb{R})\rightarrow H^{-1}(\mathbb{R})$ is an isomorphism if and only if
\begin{equation}
\ker T_{\pm ,k,\epsilon }=\{0\}.  \label{ker}
\end{equation}
We will now show that \eqref{ker} holds for all $\frac{\epsilon ^{2}}{4}<k<%
\overline{k}_\epsilon$. If $T_{\pm ,k,\epsilon }v=0$ then $v\in H^{1}(%
\mathbb{R})\cap H^{2}(\mathbb{R}\setminus \{0\})$ satisfies
\begin{gather}
-v^{\prime \prime }+kv-[6-5u_{\pm ,k,\epsilon }^{2}(x)]u_{\pm ,k,\epsilon
}^{2}(x)v=0,\quad x\neq 0,  \label{kereq} \\
v^{\prime }(0^+)-v^{\prime }(0^-)=-\epsilon v(0).  \label{ker0}
\end{gather}
Applying Theorem~3.3 of \cite{bs} to \eqref{kereq} separately on $(-\infty,0)
$ and $(0,\infty )$, and using the continuity of $v$, there exists a
constant $\alpha \in \mathbb{R}$ such that $v=\alpha |u_{\pm
,k,\epsilon}^{\prime}|$. Hence,
\begin{equation*}
v(0)=-\alpha u_{\pm ,k,\epsilon }^{\prime }(0^-)= \alpha u_{\pm
,k,\epsilon}^{\prime }(0^+)=-\alpha \frac{\epsilon}{2}u_{\pm,k,\epsilon }(0),
\end{equation*}
and since
\begin{equation*}
u_{\pm ,k,\epsilon }^{\prime \prime }(0^-)=u_{\pm ,k,\epsilon }^{\prime
\prime}(0^+) =ku_{\pm ,k,\epsilon }(0)-2u^3_{\pm ,k,\epsilon }(0)+u^5_{\pm
,k,\epsilon}(0),
\end{equation*}
it follows from \eqref{ker0} that
\begin{equation*}
4[k-2u_{\pm ,k,\epsilon }^{2}(0)+u_{\pm ,k,\epsilon }^{4}(0)]=\epsilon ^{2}.
\end{equation*}
Combining this with \eqref{uopm} yields
\begin{equation}
1-\frac{4}{3}\left( k-\frac{\epsilon ^{2}}{4}\right) =\pm \sqrt{1-\frac{4}{3}
\left( k-\frac{\epsilon ^{2}}{4}\right) }.  \label{pmequ}
\end{equation}
The `$+$' sign in \eqref{pmequ} corresponds to $u_{-,k,\epsilon }$ and
yields $k=\frac{\epsilon ^{2}}{4}$ or $k=\frac{3}{4}+\frac{\epsilon ^{2}}{4}$
, from which (i) and (iii) follow. The `$-$' sign in \eqref{pmequ}
corresponds to $u_{+,k,\epsilon }$ and yields $k=\frac{3}{4}+\frac{%
\epsilon^{2}}{4}$, so (ii) must hold. The lemma is proved.
\end{proof}

Even though the linearized operator $D_uF_\epsilon(k,u)$ becomes singular at
$(\overline{k}_\epsilon,\overline{u}_\epsilon)$, we have the following
result.

\begin{proposition}
\label{fold.prop} The set
\begin{equation}  \label{wholecurve}
\mathcal{S}:= \mathcal{S}_{-,\epsilon}\cup\{(\overline{k}_\epsilon,\overline{%
u}_\epsilon)\}\cup\mathcal{S}_{+,\epsilon}
\end{equation}
is a smooth curve in $\mathbb{R}\times H^1(\mathbb{R})$.
\end{proposition}

To prove Proposition~\ref{fold.prop} we will use a theorem of Crandall and
Rabinowitz, which enables us to reparametrize the bifurcation curve around
the point $(\overline{k}_\epsilon,\overline{u}_{\epsilon })$, where the
parametrization by $k$ breaks down. For the reader's convenience we
reproduce this result here.

\begin{theorem}[Theorem~3.2 of \protect\cite{cr73}]
\label{cranrab.thm} Let $(k_{0},u_{0})\in \mathbb{R}\times X$ where $X$ is a
Banach space and let $F$ be a continuously differentiable mapping of an open
neighborhood of $(k_{0},u_{0})$ into another Banach space $Y$. Suppose that $%
\ker D_{u}F(k_{0},u_{0})=\vect\{\eta _{0}\}$ is one-dimensional, that $\codim%
\rge D_{u}F(k_{0},u_{0})=1$, and that $D_{k}F(k_{0},u_{0})\not\in \rge %
D_{u}F(k_{0},u_{0})$. If $Z$ is a complement of $\vect\{\eta _{0}\}$ in $X$,
then the solutions of $F(k,u)=F(k_{0},u_{0})$ near $(k_{0},u_{0})$ form a
curve $(k(s),u(s))=(k_{0}+\tau (s),u_{0}+s\eta _{0}+z(s))$, where $s\mapsto
(\tau(s),z(s))\in \mathbb{R}\times Z$ is a continuously differentiable
function near $s=0$, and $\tau (0)=\dot{\tau}(0)=0,\ z(0)=\dot{z}(0)=0$.
\end{theorem}

Here, the `dot' denotes differentiation with respect to $s$.

\begin{proof}
Apply the implicit function theorem to the function $f:\mathbb{R}\times%
\mathbb{R}\times Z \to Y$ defined by
\begin{equation*}
f(s,\tau,z)=F(k_0+\tau,u_0+s\eta_0+z)
\end{equation*}
at the point $(s,\tau,z)=(0,0,0)$.
\end{proof}

\medskip

\begin{proofof}
\emph{Proposition~\ref{fold.prop}.} Firstly, since the operator $%
D_{u}F_{\epsilon }(\overline{k}_\epsilon,\overline{u}_{\epsilon })$ is
self-adjoint, it follows from \eqref{kerspan} that
\begin{equation*}
\codim\rge D_{u}F_{\epsilon }(\overline{k}_\epsilon,\overline{u}_{\epsilon
})=\dim \ker D_{u}F_{\epsilon }(\overline{k}_\epsilon,\overline{u}_{\epsilon
})=1.
\end{equation*}
Furthermore, the range of $D_{u}F_{\epsilon }(\overline{k}_\epsilon,%
\overline{u}_{\epsilon })$ is characterized by
\begin{equation*}
\rge D_{u}F_{\epsilon }(\overline{k}_\epsilon,\overline{u}_{\epsilon })=%
\Big\{v\in L^{2}(\mathbb{R}):\int_{\mathbb{R}}v\eta _{\epsilon }\,\mathrm{d}%
x=0\Big\}.
\end{equation*}
Next, we need to check that $D_{k}F_{\epsilon }(\overline{k}_\epsilon,%
\overline{u}_{\epsilon }) \not\in \rge D_{u}F_{\epsilon }(\overline{k}%
_\epsilon,\overline{u}_{\epsilon })$. But this is clear, as $%
D_{k}F_{\epsilon }(\overline{k}_\epsilon,\overline{u}_{\epsilon }) =-%
\overline{u}_{\epsilon }$ and
\begin{equation*}
\int_{\mathbb{R}}\overline{u}_{\epsilon }\eta _{\epsilon }\,\mathrm{d}x
=2\int_{0}^{\infty }\overline{u}_{\epsilon }\,\overline{u}%
_{\epsilon}^{\prime }\,\mathrm{d}x =-\overline{u}_{\epsilon }^{2}(0)<0.
\end{equation*}
It then follows from Theorem~\ref{cranrab.thm} that the solutions of %
\eqref{functequ} in a neighborhood of $(\overline{k}_\epsilon,\overline{u}%
_{\epsilon})$ form a smooth curve,
\begin{equation}
\{(k_{s},u_{s}):s\in (-\varepsilon ,\varepsilon )\}\subset \mathbb{R}\times
H^{1}(\mathbb{R})\quad (\text{for some}\ \varepsilon >0)  \label{foldcurve}
\end{equation}
such that, at $s=0$,
\begin{equation}
k_{0}=\overline{k}_\epsilon,\ \dot{k}_{0}=0,\quad u_{0}=\overline{u}%
_{\epsilon },\ \dot{u}_{0}=\eta _{\epsilon }.  \label{sderiv}
\end{equation}
Consequently, the lower and upper curves $\mathcal{S}_{-,\epsilon }$ and $%
\mathcal{S}_{+,\epsilon }$ meet smoothly at the turning point $(\overline{k}%
_\epsilon,\overline{u}_{\epsilon })$.
\end{proofof}

We can now end this section with the

\medskip

\begin{proofof}
\emph{Theorem~\ref{curve.thm}.} In view of Propositions~\ref{curves.prop}
and \ref{fold.prop}, we only need to establish the asymptotic behavior of
the upper curve as $k\searrow \frac{3}{4}$ to complete the proof of the
theorem. But this readily follows from Rabinowitz's global bifurcation
theorem \cite[Theorem~1.3]{r}. Indeed, this result states the following
alternative: either (i) the bifurcating curve meets the trivial line $%
\{(k,0) : k\in\mathbb{R}\}\subset\mathbb{R}\times H^1(\mathbb{R})$ again at
a point $(k^*,0)$ with $k^*\neq \frac{\epsilon^2}{4}$, or (ii) it is
unbounded in $\mathbb{R}\times H^1(\mathbb{R})$. In the present context,
case (i) is ruled out by the explicit form of the solutions given in Section~%
\ref{expsol.sec}. Therefore, $\mathcal{S}$ is unbounded in $\mathbb{R}\times
H^{1}(\mathbb{R})$, and so we must have
\begin{equation*}
\lim_{k\searrow 3/4}\Vert u_{+,k,\epsilon }\Vert _{H^{1}}=\infty .
\end{equation*}
Moreover, by \eqref{1stint} and \eqref{uopm}, there exists a constant $C>0$
(independent of $k$) such that
\begin{equation*}
u_{+,k,\epsilon }^{2}(x)\geqslant C(u_{+,k,\epsilon}^{\prime})^2(x),\quad
\text{for all} \ x\neq 0,\ k\in \textstyle(\frac{3}{4},\overline{k}%
_\epsilon),
\end{equation*}
which implies that $\lim_{k\searrow 3/4}\Vert u_{+,k,\epsilon
}\Vert_{H^{1}}=\infty $ if and only if
\begin{equation}
\lim_{k\searrow 3/4}\Vert u_{+,k,\epsilon }\Vert _{L^{2}}=\infty ,
\label{L2blowup}
\end{equation}
and concludes the proof of Theorem~\ref{curve.thm}.
\end{proofof}


\section{Spectral properties}

\label{morse.sec}

The purpose of this section is to prove the following spectral result, 
which is a first step towards our stability theorem.
Let $n(T_{\pm ,k,\epsilon})$ denote the number of negative eigenvalues of
the self-adjoint operator $T_{\pm ,k,\epsilon }$. 

\begin{proposition}
\label{morse.prop} The spectrum of the linear operator $T_{\pm,k,\epsilon}:%
\mathcal{D}_\epsilon\subset L^2(\mathbb{R})\to L^2(\mathbb{R})$ consists of
a finite number of simple isolated eigenvalues and a continuous part $%
[k,\infty)$. Furthermore,
\begin{equation}  \label{morse-}
n(T_{-,k,\epsilon})=1 \quad \text{for all} \ k\in\textstyle(\frac{\epsilon^2%
}{4},\overline{k}_\epsilon)
\end{equation}
and
\begin{equation}  \label{morse+}
n(T_{+,k,\epsilon})=0 \quad \text{for all} \ k\in\textstyle(\frac{3}{4},%
\overline{k}_\epsilon).
\end{equation}
\end{proposition}

\begin{proof}
As noted earlier in the proof of Lemma~\ref{nondegen.lem}, the
basic structure of the spectrum of $T_{\pm ,k,\epsilon }$ follows from
standard spectral theory, see for instance \cite{kato,s,alb}. For the simplicity
of eigenvalues, suppose that $u,v\in \mathcal{D}_{\epsilon }$ are
eigenfunctions of $T_{\pm ,k,\epsilon }$ corresponding to an eigenvalue $%
\lambda <k$. Then $(uv^{\prime }-u^{\prime }v)^{\prime }=uv^{\prime \prime
}-u^{\prime \prime }v=0$ on $\mathbb{R}\setminus \{0\}$, therefore there
exists a constant $C\in \mathbb{R}$ such that $uv^{\prime }-u^{\prime }v=C$
on $\mathbb{R}\setminus \{0\}$. However, $\lim_{|x|\rightarrow \infty
}uv^{\prime }-u^{\prime }v=0$, as $u,v\in H^{2}(\mathbb{R}\setminus \{0\})$.
Hence $C=0$ and $u,v$ are linearly dependent.

Next, \eqref{morse-} follows from a perturbation analysis similar to the
proof of Lemma~12 in \cite{lc}. The idea is first to observe that the
property holds when $\epsilon =0$. Indeed, in this case the kernel of the
linearization at the free-space soliton \eqref{freespacesol} is spanned by
its derivative $u_{-,k,0}^{\prime }$, which has a unique zero at $x=0$,
where it changes sign. Therefore, by the Sturm's oscillation theorem, zero
is the second eigenvalue of the linearization, and so $n(T_{-,k,0})=1$. With
this information at hand, we then use perturbation theory to make sure that
the second eigenvalue becomes positive for small values of $\epsilon >0$. A
first step in this direction is the smoothness of the family of operators $%
T_{-,k,\epsilon }$ at $\epsilon =0$.

\begin{lemma}
\label{holomorphy.lem} For any $\rho>0$ small enough, there exists a open
connected neighborhood $\Omega $ of the real-line segment $[0,2\sqrt{k}-\rho
]$ in $\mathbb{C}$, such that $\{T_{-,k,\epsilon }:\epsilon \in \Omega \}$
is a holomorphic family of operators.
\end{lemma}

\begin{proof}
The result follows in a similar way to Lemma~13 in \cite{lc}, and is based
on the notion of holomorphic family of unbounded operators of type (B) in
the sense of Kato, see Theorem~4.2 in chapter VII of \cite{kato}. The
argument boils down to checking that, for each fixed $x\in \mathbb{R}$, the
mapping $\epsilon \mapsto u_{-,k,\epsilon }(x)$ is holomorphic on a suitable
domain, independent of $x$. That this domain can be taken as stated in Lemma~%
\ref{holomorphy.lem} follows by a careful inspection of \eqref{expinnedsol},
and the observation that the function $\epsilon \mapsto \sqrt{%
\epsilon^{2}+(4k-\epsilon ^{2})(1-4k/3)}$ is holomorphic on the strip $%
\{\epsilon\in \mathbb{C}:\im\epsilon \in (-\sqrt{3-4k},\sqrt{3-4k})\}$.
\end{proof}

Thanks to Lemma~\ref{holomorphy.lem}, standard perturbation theory \cite%
{kato} yields two holomorphic mappings,
\begin{equation*}
\Omega \ni \epsilon \mapsto \lambda _{\epsilon }\in \mathbb{R},\quad \Omega
\ni \epsilon \mapsto w_{\epsilon }\in \mathcal{D}_{\epsilon }
\end{equation*}
such that $\lambda _{0}=0$, $w_{0}=u_{-,k,0}^{\prime }$, where $%
\lambda_{\epsilon }$ and $w_{\epsilon }$ are, respectively, the second
eigenvalue and eigenvector of $T_{-,k,\epsilon }$:
\begin{equation}
T_{-,k,\epsilon }w_{\epsilon }=\lambda _{\epsilon }w_{\epsilon },\quad
\epsilon \in \Omega .  \label{secondeigen}
\end{equation}
We will now show that $\dot{\lambda}_{0}>0$, which implies that the second
eigenvalue of $T_{-,k,\epsilon }$ is positive for small $\epsilon >0$. We
use the `dot' here to denote differentiation with respect to $\epsilon $, $%
\dot{\lambda}_{0}>0$ being the derivative of $\lambda _{\epsilon }$ with
respect to $\epsilon $ at $\epsilon =0$, and similarly for other quantities
below. Differentiating \eqref{secondeigen} with respect to $\epsilon $ at $%
\epsilon =0$ yields
\begin{equation}
-\dot{w}_{0}^{\prime \prime } +k\dot{w}_{0}-4[3-5u_{0}^{2}]u_{0}\dot{u}%
_{0}u_{0}^{\prime } -[6-5u_{0}^{2}]u_{0}^{2}\dot{w}_{0}=\dot{\lambda}_{0}%
\dot{w}_{0},  \label{wdot}
\end{equation}
where we have put $u_{0}\equiv u_{-,k,0}$ and $\dot{u}_{0}\equiv \dot{u}%
_{-,k,0}$ to simplify the notation. Observe that $w_{\epsilon }\in \mathcal{D%
}_{\epsilon }\Rightarrow \dot{w}_{0}\in \mathcal{D}_{0}=H^{2}(\mathbb{R})$.
Now, differentiating \eqref{solequ} with respect to $\epsilon $ at $%
\epsilon=0$ shows that
\begin{equation}
T_{0}\dot{u}_{0}:= -\dot{u}_{0}^{\prime \prime }+k\dot{u}_{0}
-[6-5u_{0}^{2}]u_{0}^{2}\dot{u}_{0}=\delta (x)u_{0}.  \label{diracterm}
\end{equation}
Multiplying \eqref{wdot} by $u_{0}^{\prime }$, integrating by parts and
using $T_{0}u_{0}^{\prime }=0$ then yields
\begin{equation}
\dot{\lambda}_{0}= \frac{4\int_{\mathbb{R}}[5u_{0}^{2}-3]u_{0}\dot{u}%
_{0}(u_{0}^{\prime })^2}{\int_{\mathbb{R}}(u_{0}^{\prime })^2}.
\label{lambdadot}
\end{equation}
Furthermore, straightforward calculations show that
\begin{equation*}
4[5u_{0}^{2}-3]u_{0} (u_{0}^{\prime})^2=
T_{0}(-ku_{0}+2u_{0}^{3}-u_{0}^{5})=T_{0}(-u_{0}^{\prime \prime }),
\end{equation*}
and it follows by \eqref{diracterm} that
\begin{equation*}
4\int_{\mathbb{R}}[5u_{0}^{2}-3]u_{0} (u_{0}^{\prime })^2\dot{u}
_{0}=(T_{0}(-u_{0}^{\prime \prime }),\dot{u}_{0})_{L^{2}}=(-u_{0}^{\prime
\prime },T_{0}\dot{u}_{0})_{L^{2}}=-u_{0}(0)^{\prime \prime }u(0)>0,
\end{equation*}
showing that $\dot{\lambda}_{0}$ is indeed positive. This implies that %
\eqref{morse-} holds for $\epsilon >0$ small enough. To complete the proof
of \eqref{morse-}, we invoke the continuous dependence of the first two
eigenvalues of $T_{-,k,\epsilon }$ on $\epsilon \in \lbrack 0,2\sqrt{k})$
(given by Lemma~\ref{holomorphy.lem}),
and the fact that the eigenvalues cannot cross zero unless $\epsilon =2\sqrt{%
k}$ (Lemma~\ref{nondegen.lem}~(i)).

We now turn to the proof of \eqref{morse+}. By \eqref{morse-}, the first
eigenvalue of $T_{-,k,\epsilon }$ is negative, for all $k\in \textstyle(%
\frac{\epsilon ^{2}}{4},\overline{k}_\epsilon)$. Since $\ker T_{\pm
,k,\epsilon }\neq\{0\}\Leftrightarrow k=\overline{k}_\epsilon$ by Lemma~\ref%
{nondegen.lem}, we only need to show that the first eigenvalue of $T_{\pm
,k,\epsilon }$ crosses zero at $(\overline{k}_\epsilon,\overline{u}%
_{\epsilon })$. Using the parametrization \eqref{foldcurve}, and denoting by
$\mu _{s}$ the first eigenvalue along the curve, this amounts to showing
that $\dot{\mu}_{0}\neq 0$. The first eigenvalue and eigenfunction $\mu
_{s},v_{s}$ satisfy $\mu _{0}=0,\ v_{0}=\eta _{\epsilon }$, and
\begin{equation*}
v_{s}\in \mathcal{D}_{\epsilon },\quad -v_{s}^{\prime
\prime}+k_sv_{s}-[6u_{s}^{2}-5u_{s}^{4}]v_{s}=\mu _{s}v_{s}, \quad s\in
(-\varepsilon,\varepsilon ).
\end{equation*}
In view of \eqref{sderiv}, differentiating with respect to $s$ and letting $%
s=0$ yields
\begin{equation}
-\dot{v}_{0}^{\prime \prime }+\overline{k}_\epsilon\dot{v}_{0} -4[3-5%
\overline{u}_{\epsilon }^{2}]\overline{u}_{\epsilon }\eta _{\epsilon }^{2}
-[6-5\overline{u}_{\epsilon }^{2}]\overline{u}_{\epsilon }^{2}\dot{v}_{0}=%
\dot{\mu}_{0}\eta_{\epsilon },  \label{vdot}
\end{equation}
where the `dot' now denotes again differentiation with respect to $s$.
Multiplying both sides of \eqref{vdot} by $\eta _{\epsilon }$, integrating
by parts and using $D_{u}F_{\epsilon }(\overline{k}_\epsilon,\overline{u}%
_{\epsilon})\eta _{\epsilon }=0$ yields
\begin{equation}
\dot{\mu}_{0}= \frac{4\int_{\mathbb{R}}[5\overline{u}_{\epsilon }^{2}-3]%
\overline{u}_{\epsilon }\eta _{\epsilon }^{3}} {\int_{\mathbb{R}%
}\eta_{\epsilon }^{2}}.  \label{mudot}
\end{equation}
We were not able to find an analytical argument showing that
\begin{equation*}
f(\epsilon ):=\int_{\mathbb{R}}[5\overline{u}_{\epsilon }^{2}-3]\overline{u}%
_{\epsilon }\eta _{\epsilon }^{3}= 2\int_{0}^{\infty }[5\overline{u}%
_{\epsilon }^{2}-3] \overline{u}_{\epsilon }|\overline{u}_{\epsilon
}^{\prime}|^3\neq 0.  \label{crosszero}
\end{equation*}
\begin{figure}[h]
\includegraphics[width=80mm]{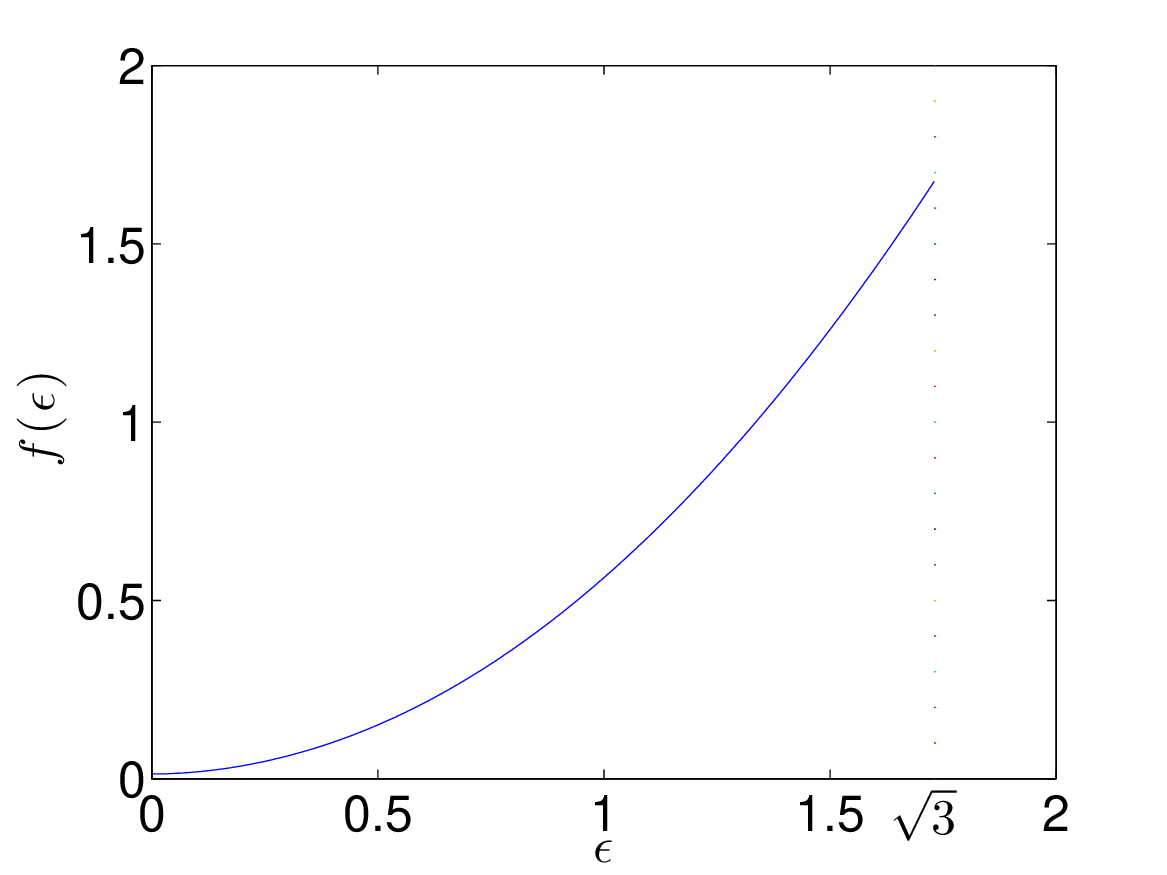}
\caption{The graph of $f(\protect\epsilon )$}
\label{fig1a}
\end{figure}
However, numerical computation of this integral (Fig.~\ref{fig1a}) clearly
shows that it is positive for all values of $\epsilon \in (0,\sqrt{3})$,
which concludes the proof of Proposition~\ref{morse.prop}. 
\end{proof}


\section{Stability}

\label{stab.sec}

We consider the stability of the bound states
\begin{equation}  \label{boundstates}
\psi _{\pm ,k,\epsilon }(x,z)=\mathrm{e}^{ikz}u_{\pm ,k,\epsilon }(x)
\end{equation}
with respect to perturbations of the initial soliton profile, $%
u_{\pm,k,\epsilon }$, in $H^{1}(\mathbb{R})$. Let us first remark that the
Cauchy problem associated with \eqref{nls} is globally well posed in $H^{1}(%
\mathbb{R})$, see \cite{cazenave}. That is, for any initial profile $\psi
(\cdot,0)\in H^{1}(\mathbb{R})$, there exists a unique continuous map $%
z\mapsto \psi (x,z)\in H^{1}(\mathbb{R})$, defined for all $z\in\real$, such that
$\psi(x,z)$ satisfies \eqref{nls}.

We will now define precisely what we mean by the stability of the bound
states of \eqref{nls}. It is well known that, due to the $U(1)$%
-invariance of \eqref{nls}, the appropriate notion of stability in this
context is that of orbital stability.

\begin{definition}
\rm
\label{orbstab.def} 
We say that the bound state $\psi _{k}(x,z)=%
\mathrm{e}^{ikz}u_{k}(x)$ is \emph{orbitally stable} if
\begin{equation*}
\text{for all} \ \varepsilon >0 \ \text{there exists} \ \delta >0\ \text{%
such that}
\end{equation*}
for any solution $\varphi (x,z)$ of (\ref{nls}) with initial data $%
\varphi(\cdot,0)\in H^{1}(\mathbb{R})$ there holds
\begin{equation*}
\Vert \varphi (\cdot,0)-u_{k}\Vert _{H^{1}}\leqslant \delta \implies
\inf_{\theta \in \mathbb{R}}\Vert \varphi (\cdot ,z)-\mathrm{e}%
^{i\theta}u_{k}\Vert _{H^{1}} \leqslant \varepsilon \quad \text{for all}\
z\geqslant 0.
\end{equation*} 
\end{definition}

Introducing the \emph{orbit} $\Theta (u_{k}):=\{\mathrm{e}%
^{i\theta}u_{k}:\theta \in \mathbb{R}\}$ of $u_{k}$, the above statement
can be rephrased as
\begin{equation*}
\dist(\varphi (\cdot,0),u_{k})\leqslant \delta \implies \dist%
(\varphi (\cdot ,z),\Theta (u_{k}))\leqslant \varepsilon \quad \text{for
all}\ z\geqslant 0.
\end{equation*}
Here we have introduced $\dist(u,v):=\Vert u-v\Vert _{H^{1}},\ u,v\in H^{1}(%
\mathbb{R})$, and the distance from a point $u\in H^{1}(\mathbb{R})$ to a
set $B\subset H^{1}(\mathbb{R})$ is defined as $\dist%
(u,B):=\inf_{v\in B}\dist(u,v)$.

A general theory of orbital stability for Hamiltonian systems invariant
under the action of a one-parameter group has been established in \cite{gss}.
More recently,
the theory has been revisited and extended in \cite{stuart2008} (with a special focus on nonlinear
Schr\"odinger equations) and in \cite{lnm}, where it is formulated in
a natural geometric framework. We shall now briefly outline how \eqref{nls} can be
interpreted as a Hamiltonian system, and state the key stability criteria we will use to prove
that the whole solution curve $\mathcal{S}$
consists of orbitally stable bound states. 

First, identify $H^1(\real,\complex)$ with $X:=H^1(\real,\real)\times H^1(\real,\real)$ by writing
$\psi=(\re\psi,\im\psi)\in X$ for all $\psi\in H^1(\real,\complex)$. 
We shall henceforth merely write $H^1$ for $H^1(\real,\real)$.
Identifying $L^2$ with its dual space,
consider the variational triple $H^1\subset L^2 \subset H^{-1}$, and $I:H^1\hookrightarrow H^{-1}$ 
the injection. Introducing the function $f(x,s)=\ep\delta(x)+2s-s^2, \ x\in\real, \ s\ge0,$ \eqref{nls} becomes
$$
i\psi_z=-\psi_{xx}-f(x,|\psi|^2)\psi,
$$
which can then be cast as
\begin{equation}\label{ham}
\frac{\dif}{\dif z}\psi(z)=JE'(\psi(z)),
\end{equation}
where  
$$
J=
\begin{pmatrix}
0 & -I\\
I & 0
\end{pmatrix}
$$
and the energy $E:X\to\real$ is given by
\begin{equation}\label{energie}
E(\psi)=\frac12\int_\real |\psi'(x)|^{2}\diff x-\frac12\int_\real \int_0^{|\psi|^2}f(x,s)\diff s\diff x.
\end{equation}
Of course, the energy is a conserved quantity. Namely, for any solution $\psi(x,z)$ of \eqref{nls} we have
$E(\psi(\cdot,z))=E(\psi(\cdot,0))$ for all $z\ge0$. Another important conserved quantity is the power
of the beam, given by
$$
Q(\psi)=\frac{1}{2}\int_\real |\psi(x)|^2\diff x.
$$
For the following discussion it is important to observe that $E,Q\in C^2(X,\real)$.

In this formalism, bound states take the form $\psi(x,z)=T(kz)\ffi$ for some $\ffi \in X$ and
$$
T(\theta)=
\begin{pmatrix}
\cos\theta & -\sin\theta\\
\sin\theta & \cos\theta
\end{pmatrix}, \quad \theta\in\real.
$$
Furthermore, the stationary equation \eqref{solequ} now reads
\begin{equation}\label{statham}
E'(\ffi)+kQ'(\ffi)=0
\end{equation}
for some real $\ffi=(u,0)$.

The Hamiltonian system \eqref{ham} is invariant under the action of the group 
$\{T(\theta)\}_{\theta\in\real}$. 
This corresponds to the invariance of \eqref{nls} with respect to multiplication by a phase factor
$\e^{i\theta}$. It becomes apparent that the notion of orbital stability defined above is precisely stability
modulo the action of this group. 

Given $k\in(\frac{\ep^2}{4},\overline{k}_\ep]$, the orbital stability of a corresponding solution $\ffi_k=(u_k,0)$ of \eqref{statham} can be proved by using the Lyapunov functional $L_k:X\to\real$,
$$
L_k(\ffi)=E(\ffi)+k Q(\ffi).
$$
The stability of the bound state  $\ffi_k=(u_k,0)$ then relies on
a coercivity property of $L_k$ that can be formulated in terms of the second derivative
$D^2L_k(\ffi_k):X\to X^*$, where $X^*=H^{-1}\times H^{-1}$. 
Using the stationary equation \eqref{statham}, we have
$$
D^2L_k(\ffi_k)=E''(\ffi_k)+k Q''(\ffi_k)=
\begin{pmatrix}
L^+(k,u_k) & 0\\
0 & L^-(k,u_k)
\end{pmatrix},
$$
where
$$
L^+(k,u_k)=-\frac{\dif^2}{\dif x^2}+k-f(x,u_k^2)-2\partial_sf(x,u_k^2)u_k^2 
=-\frac{\dif^2}{\dif x^2}+k-\ep\delta(x)-6u_k^2(x)+5u_k^4(x)
$$
and
$$
L^-(k,u_k)=-\frac{\dif^2}{\dif x^2}+k-f(x,u_k^2)
=-\frac{\dif^2}{\dif x^2}+k-\ep\delta(x)-2u_k^2(x)+u_k^4(x).
$$
As usual, for $v\in H^1$, $-v''$ is interpreted as an element of $H^{-1}$ through 
$\langle -v'',w\rangle_{H^{-1}\times H^1} = \langle v',w'\rangle_{L^2}$ for all $w\in H^1$.
Various forms of the required coercivity condition are given in \cite{stuart2008}. 
They all express the fact that the Hessian $D^2L_k(\ffi_k)$ is positive definite on the
codimension 2 subspace of $H^1$ orthogonal to the orbit $\Theta (u_{k})$ 
and parallel to the tangent space
to the surface $Q(u)=Q(u_k)$ at $u=u_k$; see condition (SC) in \cite[p.~349]{stuart2008} 
and Equ.~(126) in \cite{lnm}. We will use here condition (SC**) formulated in the NLS context in 
\cite[pp.~379-380]{stuart2008} as follows: 
There exists $\delta>0$ s.t.
\begin{equation}\label{coerc}
\langle L^+(k,u_k)v,v\rangle_{H^{-1}\times H^1}\ge \delta \Vert v\Vert_{L^2}^2 \ \text{and} \
\langle L^-(k,u_k)v,v\rangle_{H^{-1}\times H^1}\ge \delta \Vert v\Vert_{L^2}^2 \
\text{for all} \ v\in H^1(\real,\real) \ \text{s.t.} \int_\real vu_k \diff x=0.
\end{equation}

As earlier, we can interpret $L^\pm(k,u_k)$ as self-adjoint operators acting in $L^2$.
Remarking that the solution $u_k>0$ satisfies $L^-(k,u_k)u_k$,
it follows by standard spectral theory \cite{kato,s,alb} that $\ker L^-(k,u_k)=\vect\{u_k\}$, and the
spectrum of $L^-(k,u_k)$ consists of the eigenvalue zero, possibly some positive eigenvalues,
and the essential spectrum $[k,\infty)$. Furthermore, 
the spectrum of $L^+(k,u_{\pm,k,\ep})=T_{\pm,k,\epsilon}$ is known from Lemma~\ref{morse.prop}, and
we have to distinguish three different cases:
\begin{itemize}
\item[(I)] For $k\in(\frac{3}{4},\overline{k}_\ep)$, $L^+(k,u_{+,k,\ep})$ has strictly positive spectrum.
\item[(II)] For $k\in(\frac{\ep^2}{4},\overline{k}_\ep)$, $L^+(k,u_{-,k,\ep})$ has exactly one negative 
eigenvalue of multiplicity 1, and the rest of its spectrum is strictly positive.
\item[(III)] At the fold bifurcation point, $k=\overline{k}_\ep$, $L^+(\overline{k}_\ep,\overline{u}_\ep)$
has zero as a simple eigenvalue, 
$\ker L^+(\overline{k}_\ep,\overline{u}_\ep)=\vect\{|\overline{u}_\ep'|\}$, and the rest of the spectrum 
is strictly positive.
\end{itemize}

The spectral scenario (II) is the most common one in the NLS literature. 
In this case, the coercivity condition \eqref{coerc} can be derived 
(see for instance \cite[Proposition~9]{lnm}) as a consequence of the 
so-called \emph{VK condition}:\footnote{%
The monotonicity condition in \eqref{VKcondition} seems to have first been
formulated by Vakhitov and Kolokolov in \cite{vk}, and so is often referred
to as the `Vakhitov-Kolokolov condition' (VK condition for short).}
\begin{equation}
\frac{\mathrm{d}}{\mathrm{d}k}\Vert u_{\pm ,k,\epsilon }\Vert
_{L^{2}}^{2}>0,\quad k\in (\textstyle\frac{\epsilon ^{2}}{4},\overline{k}%
_\epsilon).  \label{VKcondition}
\end{equation}
In the physics literature, the VK condition
is often regarded as a criterion for stability on its own, taking no account of the underlying spectral
landscape. It may therefore seem surprising to the more physical reader that the solution 
curve $\mathcal{S}_{+,\epsilon }$ is indeed stable, even though it violates the VK
criterion. A formal justification was
nevertheless carried out by Yang in \cite{y}, where fold bifurcations
for general nonlinear Schr\"{o}dinger equations are studied.

We can now prove the stability theorem.

\begin{theorem}
\label{stable.thm} Let $\epsilon\in(0,\sqrt{3})$, and $\mathcal{S}$ be
defined by \eqref{wholecurve}. Then, for all $(k,u)\in \mathcal{S}$, $\psi
(x,z)=\mathrm{e}^{ikz}u(x)$ is an orbitally stable solution of \eqref{nls}.
\end{theorem}

\begin{proof}
We first address cases (I) and (II), that is, the stability of the solutions belonging to the pieces of
curve $\mathcal{S}_{\pm ,\epsilon }$. In case (I), the verification of \eqref{coerc} is straightforward,
so $\mathcal{S}_{+,\epsilon }$ is indeed stable.
For the stability of $\mathcal{S}_{-,\epsilon }$ we will prove that the function
\begin{equation}
(\textstyle\frac{\epsilon ^{2}}{4},\overline{k}_\epsilon)\ni k\mapsto \Vert
u_{-,k,\epsilon }\Vert _{L^{2}}  \label{monotonemap}
\end{equation}
is strictly increasing. 
Firstly, for $k\in (\frac{3}{4},\overline{k}_\epsilon)$, we find using
Mathematica that\footnote{%
It turns out that the expressions for $\frac{\mathrm{d}}{\mathrm{d}k}\Vert
u_{\pm,k,\epsilon }\Vert _{L^{2}}^{2}$ are much simpler than those for $%
\Vert u_{\pm ,k,\epsilon }\Vert _{L^{2}}^{2}$ in the regime $k\in (\frac{3}{4%
},\overline{k}_\epsilon)$.}
\begin{equation}
\frac{\mathrm{d}}{\mathrm{d}k}\Vert u_{-,k,\epsilon }\Vert _{L^{2}}^{2}=
\frac{\frac{2\sqrt{3}\epsilon }{\sqrt{3+\epsilon ^{2}-4k}}-\frac{3}{\sqrt{k}}%
}{4k-3}\quad \text{and}\quad \frac{\mathrm{d}}{\mathrm{d}k}\Vert
u_{+,k,\epsilon }\Vert _{L^{2}}^{2} =-\frac{\frac{2\sqrt{3}\epsilon }{\sqrt{%
3+\epsilon ^{2}-4k}}+\frac{3}{\sqrt{k}}}{4k-3}.  \label{explicitslopes}
\end{equation}
It follows that $\Vert u_{-,k,\epsilon }\Vert _{L^{2}}$ is indeed
increasing, while $\Vert u_{+,k,\epsilon }\Vert _{L^{2}}$ is decreasing, for
$k\in (\frac{3}{4},\overline{k}_\epsilon)$. We also observe explicitly here
that $\lim_{k\searrow \frac{3}{4}}\frac{\mathrm{d}}{\mathrm{d}k}\Vert
u_{+,k,\epsilon }\Vert _{L^{2}}^{2}=-\infty$, which is consistent with
Theorem~\ref{curve.thm}.

For $k<3/4$, a straightforward calculation using \eqref{expinnedsol} shows
that
\begin{equation}
\Vert u_{-,k,\epsilon }\Vert _{L^{2}}^{2}=\sqrt{3}\log \varphi
_{\epsilon}(k)\quad \text{where}\quad \varphi _{\epsilon }(k):=\frac{\sqrt{3}%
\epsilon +\sqrt{3\epsilon ^{2}+(4k-\epsilon ^{2})(3-4k)} +\big(\sqrt{3}+2%
\sqrt{k}\big)\big(2\sqrt{k}-\epsilon \big)}{\sqrt{3}\epsilon +\sqrt{%
3\epsilon^{2}+(4k-\epsilon ^{2})(3-4k)}+\big(\sqrt{3}-2\sqrt{k}\big)\big(2%
\sqrt{k}-\epsilon \big)}.  \label{phi}
\end{equation}
Differentiation then yields
\begin{equation}
\frac{\mathrm{d}}{\mathrm{d}k}\varphi _{\epsilon }(k)=8\sqrt{k}\frac{\sqrt{3}%
\sqrt{3\epsilon ^{2} +(4k-\epsilon ^{2})(3-4k)}+2\sqrt{k}\big(%
3+\epsilon^{2}-2\epsilon \sqrt{k}\big)}{\sqrt{3\epsilon ^{2} +(4k-\epsilon
^{2})(3-4k)}\Big(\sqrt{3}\epsilon +\sqrt{3\epsilon ^{2}+(4k-\epsilon
^{2})(3-4k)} +\big(\sqrt{3}-2\sqrt{k}\big)\big(2\sqrt{k}-\epsilon \big)\Big)%
^{2}},  \label{dphi}
\end{equation}
where we observe that
\begin{equation*}
k<\frac{3}{4}\implies 3+\epsilon ^{2}-2\epsilon \sqrt{k}> \big(\sqrt{3}%
-\epsilon \big)^{2}+\sqrt{3}\epsilon >0.
\end{equation*}
Therefore, $\frac{\mathrm{d}}{\mathrm{d}k}\varphi _{\epsilon}(k)>0$, so $%
\Vert u_{-,k,\epsilon }\Vert _{L^{2}}$ is also increasing for all $k\in (%
\frac{\epsilon ^{2}}{4},\frac{3}{4})$. We have thus proved that the curves 
$\mathcal{S}_{\pm ,\epsilon }$ are both stable.

We finally consider case (III). 
To prove the stability of the solution $(\overline{k}_\epsilon,\overline{u}_{\epsilon })$
we show directly that \eqref{coerc} holds. Since 
$\ker L^-(\overline{k}_\epsilon,\overline{u}_{\epsilon })=\vect\{\overline{u}_{\epsilon }\}$,
the second condition in \eqref{coerc} is clearly satisfied. Similarly, in view of (III),
$$
\langle L^+(\overline{k}_\epsilon,\overline{u}_{\epsilon })v,v\rangle_{H^{-1}\times H^1}\ge 
\delta \Vert v\Vert_{L^2}^2 \ 
\text{for all} \ v\in H^1(\real,\real) \ \text{s.t.} \int_\real v|\overline{u}_\ep'| \diff x=0.
$$ 
But $\int_\real|\overline{u}_\ep'|\overline{u}_\ep\diff x>0$, so denoting by $P$ the
projection onto the orthogonal space to $|\overline{u}_\ep'|$ in $L^2(\real)$, there exists
$a>0$ such that
$$
\Vert P v\Vert_{L^2}^2\ge a  \Vert v\Vert_{L^2}^2 \ \text{for all} \ v\in H^1(\real,\real)  \
\text{s.t.}  \int_\real v\overline{u}_\ep \diff x=0.
$$
Then, since 
$L^+(\overline{k}_\epsilon,\overline{u}_{\epsilon }):
\mathcal{D}_\epsilon\subset L^2(\mathbb{R})\to L^2(\mathbb{R})$ is self-adjoint with 
$L^+(\overline{k}_\epsilon,\overline{u}_{\epsilon })|\overline{u}_\ep'|=0$, it follows that
$$
\langle L^+(\overline{k}_\epsilon,\overline{u}_{\epsilon })v,v\rangle_{H^{-1}\times H^1}
=\langle L^+(\overline{k}_\epsilon,\overline{u}_{\epsilon })Pv,Pv\rangle_{H^{-1}\times H^1}
\ge \delta{a} \Vert v\Vert_{L^2}^2 \ 
\text{for all} \ v\in H^1(\real,\real) \ \text{s.t.} \int_\real v\overline{u}_\ep \diff x=0.
$$ 
Hence, the second condition in \eqref{coerc} is also satisfied.
The proof is complete.
\end{proof}

\begin{remark}
\rm
One can deduce from \eqref{explicitslopes} and \eqref{phi}--%
\eqref{dphi} that
\begin{equation*}
\lim_{k\searrow3/4}\frac{\mathrm{d}}{\mathrm{d} k}\Vert
u_{-,k,\epsilon}\Vert_{L^2}^2= \lim_{k\nearrow3/4}\frac{\mathrm{d}}{\mathrm{d%
} k}\Vert u_{-,k,\epsilon}\Vert_{L^2}^2 =\sqrt{3}\left(\frac{1}{\epsilon^2}%
+\frac13\right),
\end{equation*}
showing that the slopes calculated from the solutions with $k<3/4$ and with $%
k>3/4$ indeed match where the two portions of $\mathcal{S}_{-,\epsilon}$
meet.
\end{remark}


\section{Numerics}

\label{num.sec}

Hereafter we present a numerical method, which we used for computing
solutions of Eq.~\eqref{solequ}. This was helpful to understand the behavior
of solutions before we had found their explicit representations. The method
is based on the \emph{continuous normalized gradient flow}, which was
studied and implemented in \cite{bao} in the context of the NLS equation
with a cubic nonlinearity.

\subsection{The numerical scheme}

We look for a minimizer of the energy
\begin{equation}
E(u)=\frac{1}{2}\Big\{\Vert u_{x}\Vert _{L^{2}}^{2}-\epsilon
|u(0)|^{2}-\Vert u\Vert _{L^{4}}^{4}+\frac{1}{3}\Vert u\Vert _{L^{6}}^{6}%
\Big\},  \label{energy}
\end{equation}
with a given power constraint
\begin{equation}
\Vert u\Vert _{L^{2}}=a>0.  \label{constraint}
\end{equation}
The minimizer is then a solution of \eqref{solequ} which can be interpreted
as a (nonlinear) eigenfunction with eigenvalue
\begin{equation*}
k=\frac{-\Vert u_{x}\Vert _{L^{2}}^{2}+\epsilon |u(0)|^{2}+2\Vert u\Vert
_{L^{4}}^{4}-\Vert u\Vert _{L^{6}}^{6}}{\Vert u\Vert _{L^{2}}^{2}}.
\label{eigenvalue}
\end{equation*}

In the physics literature, this method is known as imaginary time
propagation ($z\rightarrow -it$) \cite{Im-time}
\begin{equation*}
u_{t}=-\frac{\delta E(u)}{\delta u}=u_{xx}+\epsilon \delta
(x)u+2|u|^{2}u-|u|^{4}u.  \label{imtime}
\end{equation*}
Thus, in order to solve \eqref{solequ}, we introduce the imaginary time and
iterate in this time. After each time step, we renormalize the solution so
as to maintain the constraint (\ref{constraint}). The discretization of %
\eqref{imtime} is done by means of semi-implicit backward Euler central
differences.

Let us consider the time sequence $t_{0}<t_{1}<t_{2}<\dots <t_{n}$, with
time step $\mathrm{d}t=t_{n}-t_{n-1}$, and space grid $x_{j}=x_{0}+jh_{x}$
with $j=0,1,2,\dots ,J$, where we solve the equation on $[x_{0},x_{J}]$ with
$J$ grid points and the mesh size $h_{x}=(x_{J}-x_{0})/J$. The discrete
solution is denoted by $u_{j}^{n}=u(t^{n},x_{j})$ and $j_{0}$ is the index
for which $x_{j_{0}}=0$. At $x_{j_{0}}$ we use the properties of $u$ in
Proposition~\ref{basic.prop}~(v), in the discrete form:
\begin{equation*}
u_{j_{0}+1}^{n}=\Big(1-\frac{h_{x}\cdot \epsilon }{2}\Big) u_{j_{0}}^{n
},\quad u_{j_{0}-1}^{n }=\Big(1-\frac{h_{x}\cdot \epsilon }{2}\Big)%
u_{j_{0}}^{n}  \label{discrete0}
\end{equation*}

On $[t_n,t_{n+1}]$ we solve:
\begin{align*}
\frac{u^{*}_j -u^n_j}{\mathrm{d} t} &=\frac{u^{*}_{j+1}-2u^{*}_j+u^{*}_{j-1}%
}{h_x^2}+ 2 (u^n_j)^2 u^{*}_j- (u^n_j)^4 u^{*}_j & \quad \mbox{ for }
0\leqslant j<j_0-1 \mbox{ and } j_0+1 <j\le J; \\
\frac{u^{*}_{j} -u^n_{j}}{\mathrm{d} t}&=\frac{u^{*}_{j}(\frac{2}{2-hx \cdot
\epsilon}-2)+u^{*}_{j}}{h_x^2}+ 2 (u^n_j)^2 u^{*}_j- (u^n_j)^4 u^{*}_j &
\quad \mbox{ for } j=j_0-1; \\
u^*_{j_0}&=\frac{u^*_{j_0-1}}{(1-\frac{h_x \cdot \epsilon}{2})} & \quad %
\mbox{ for } j= j_0; \\
\frac{u^{*}_j -u^n_j}{\mathrm{d} t} &=\frac{u^{*}_{j+1}+u^{*}_j(\frac{2}{%
2-h_x \cdot \epsilon}-2)}{h_x^2}+ 2 (u^n_j)^2 u^{*}_j- (u^n_j)^4 u^{*}_j &
\quad \mbox{ for } j=j_0+1; \\
u^{n+1}_j&=\frac{a\cdot u^{*}_j}{\|u^{*}\|_2} & \quad \text{for all} \ j.
\end{align*}

\subsection{Numerical simulations}

In this section we compare the discretized solution $u_{j}^{n}$ with the
exact one for different values of the parameters. We solve the equation on $%
[-40,40]$, with $J=3200$ grid points and time step $\mathrm{d}t=10^{-4}$
(thus $x_{0}=-40,x_{J}=40$, $j_{0}=1600$ and $h_{x}=1/40$). For fixed $%
\epsilon $ we draw the bifurcation diagram for the power of the exact
solution \eqref{expinnedsol}--\eqref{regime2}, i.e., its norm $\Vert
u\Vert_{L^{2}}$, and pick up values $a_{1},a_{2},\dots ,a_{6}$ of the power,
see Fig.~\ref{curve1.fig}--\ref{curve3.fig}. Then we calculate the
discretized solution $u_{j}^{n}$ with fixed power $a_{l}$ ($l=1,2,\dots ,6$),
and compare it to the exact solution with the corresponding $k$.

\begin{figure}[th]
\includegraphics[width=80mm]{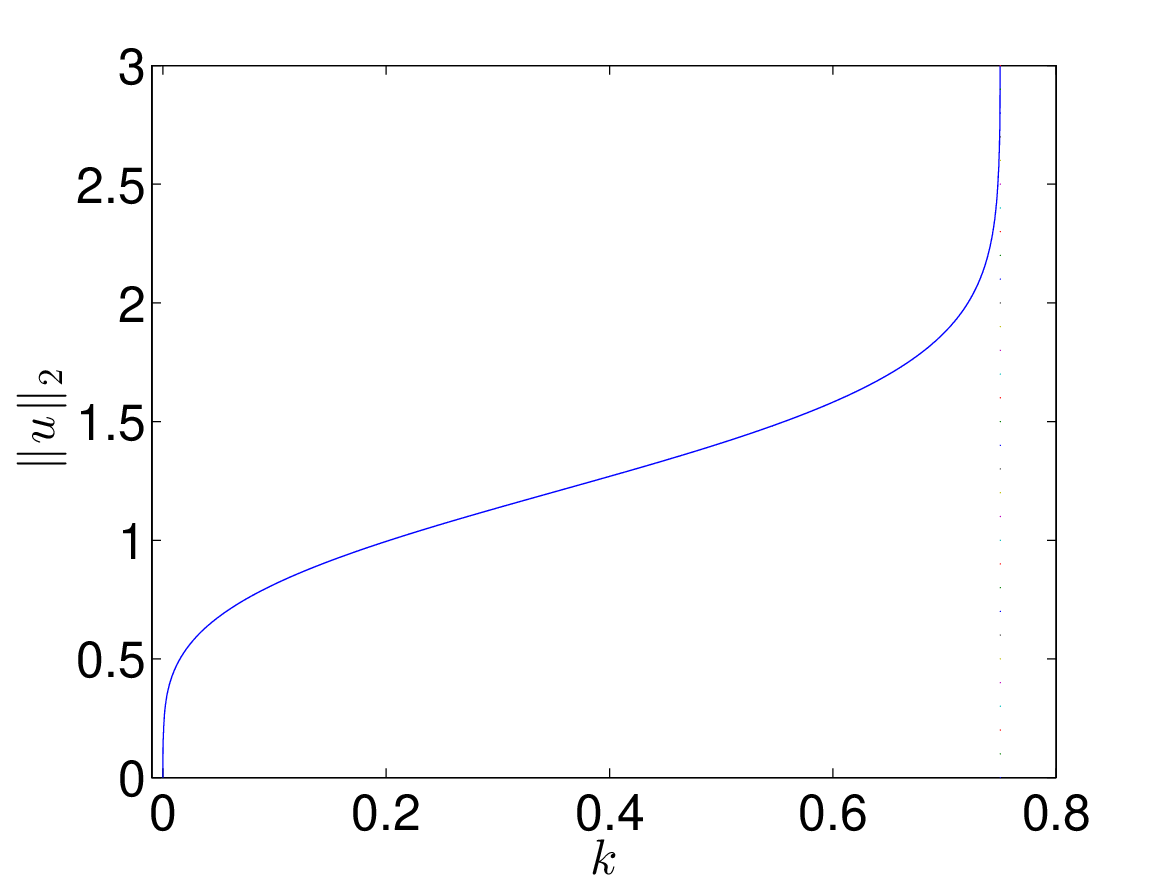}
\caption{For $\protect\epsilon =0$, we plot $\Vert u\Vert _{L^{2}}$ against $%
k$, using the explicit solutions $u_{-,k,0}$ obtained in Section~\protect\ref%
{expsol.sec}.}
\label{curve0.fig}
\end{figure}

\begin{figure}[th]
\includegraphics[width=80mm]{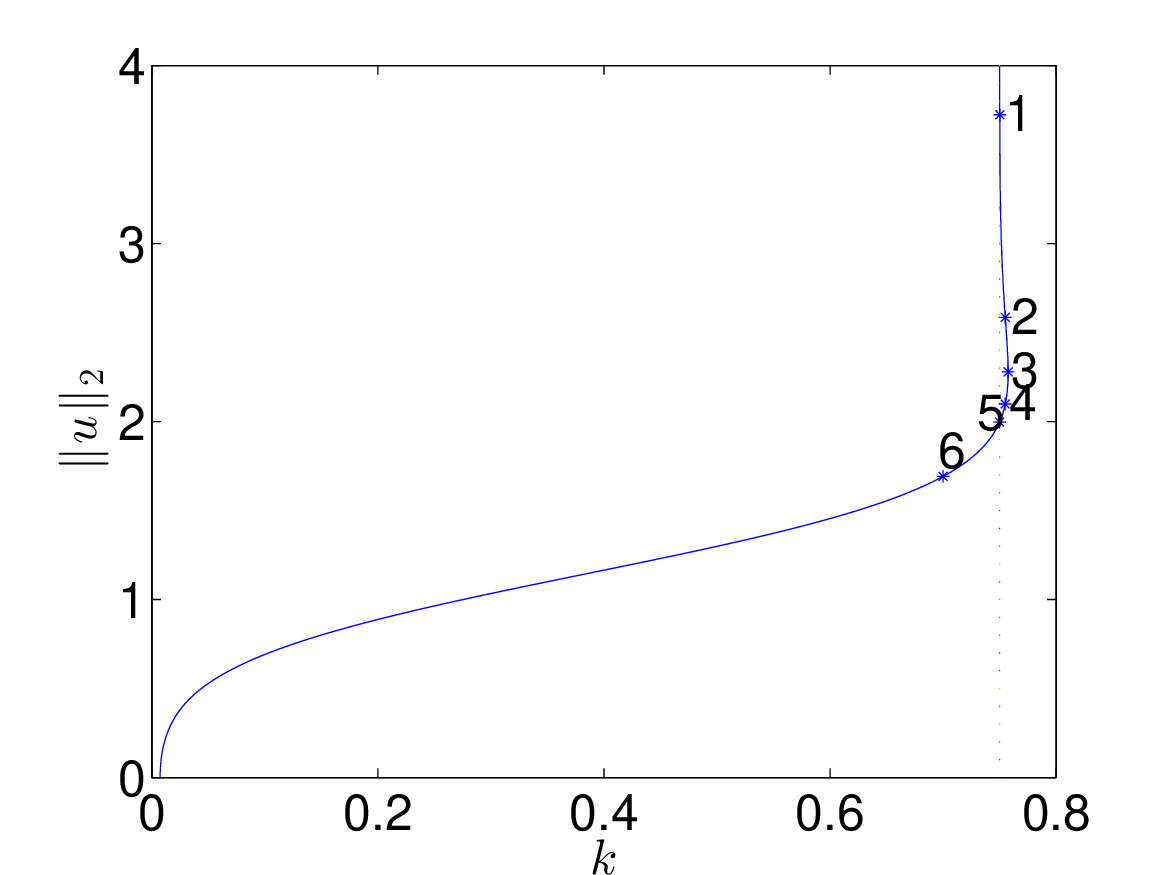} %
\includegraphics[width=80mm]{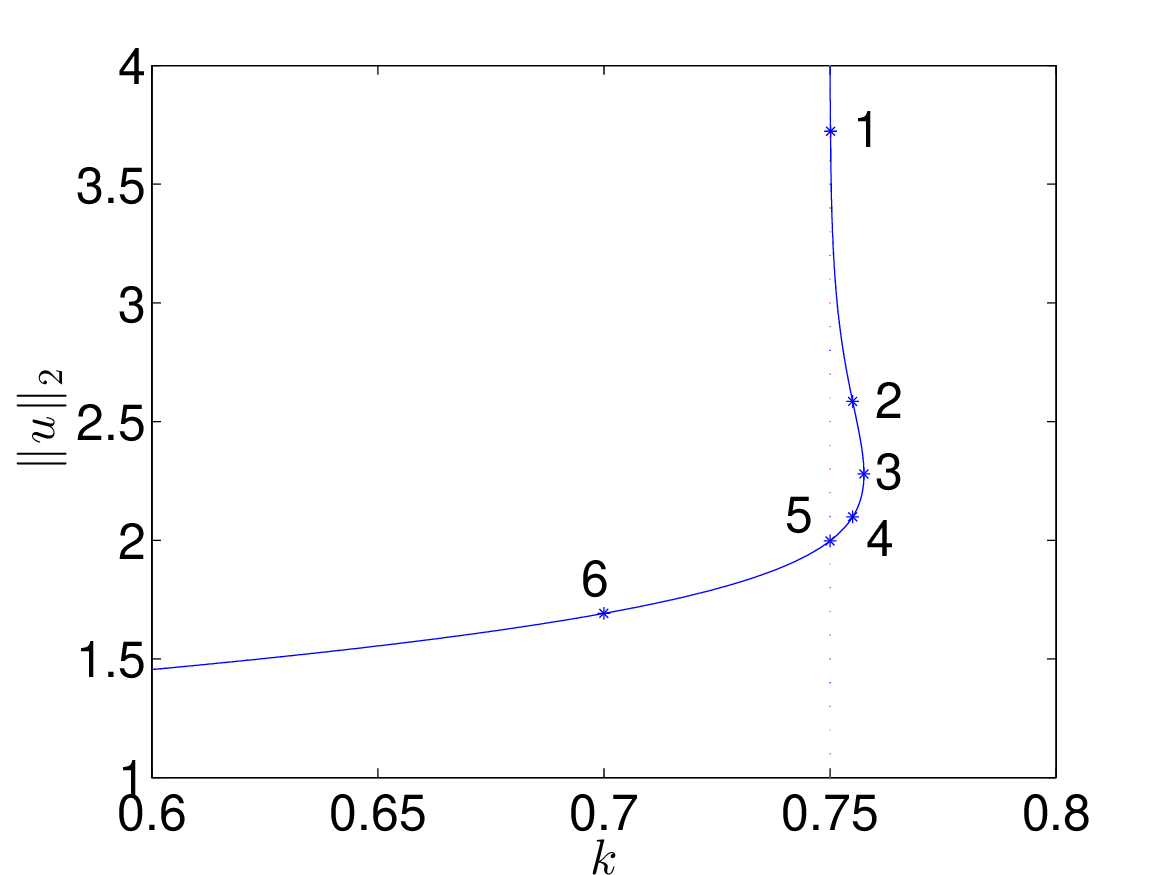}
\caption{For $\protect\epsilon =0.1\protect\sqrt{3}$, we plot $\Vert u\Vert
_{L^{2}}$ against $k$, using the explicit solutions $u_{\pm ,k,\protect%
\epsilon } $ obtained in Section~\protect\ref{expsol.sec}. In the second
plot we zoomed in, to have a closer view of the fold bifurcation.}
\label{curve1.fig}
\end{figure}

\begin{figure}[th]
\includegraphics[width=80mm]{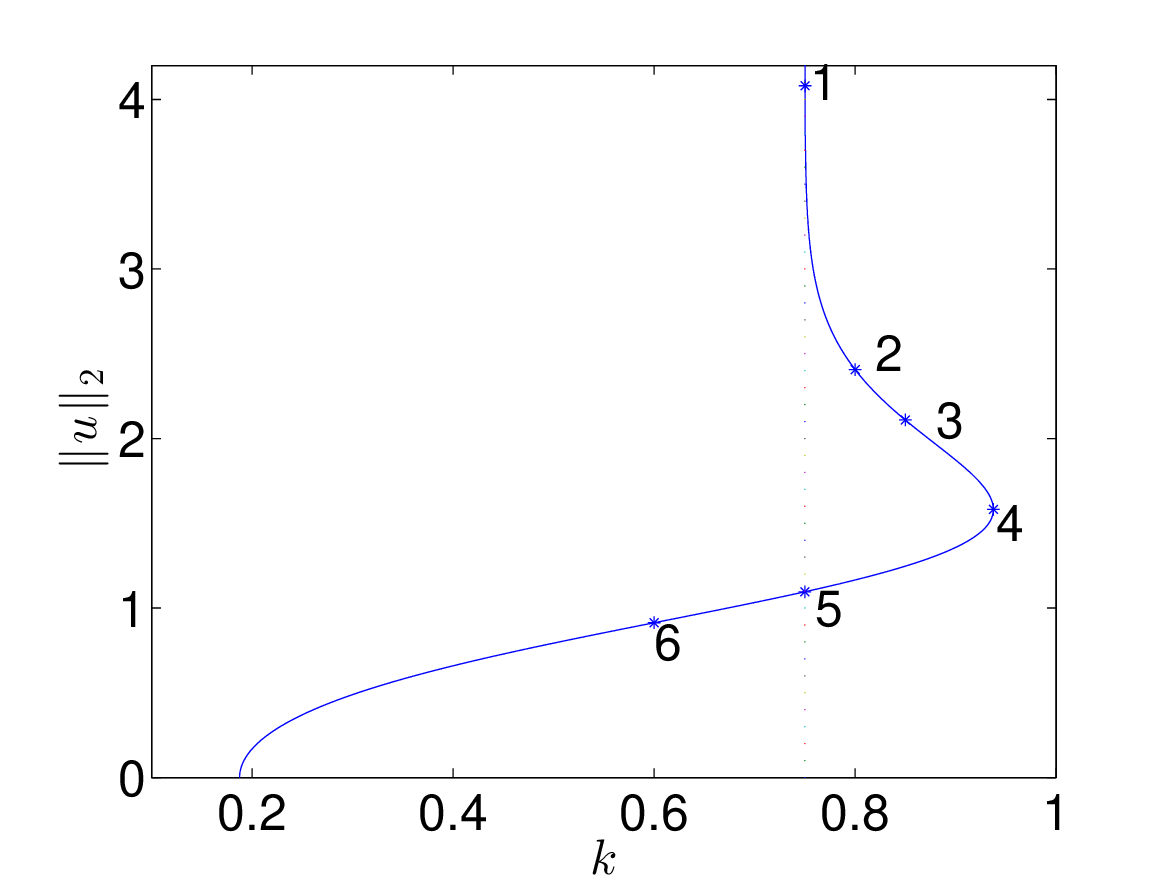}
\caption{For $\protect\epsilon =0.5\protect\sqrt{3}$, we plot $\Vert u\Vert
_{L^{2}}$ against $k$, using the explicit solutions $u_{\pm ,k,\protect%
\epsilon } $ obtained in Section~\protect\ref{expsol.sec}.}
\label{curve2.fig}
\end{figure}

\begin{figure}[th]
\includegraphics[width=80mm]{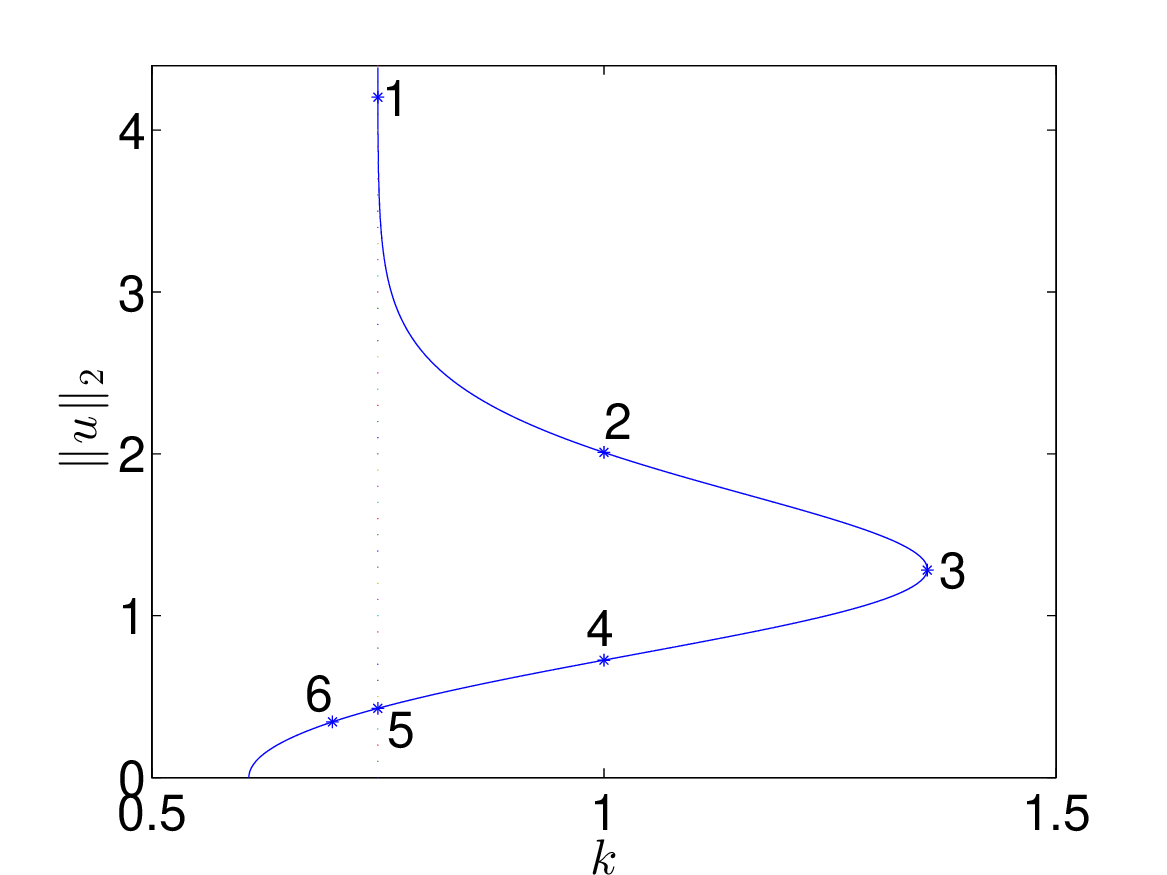}
\caption{For $\protect\epsilon =0.9\protect\sqrt{3}$, we plot $\Vert u\Vert
_{L^{2}}$ against $k$, using the explicit solutions $u_{\pm ,k,\protect%
\epsilon } $ obtained in Section~\protect\ref{expsol.sec}.}
\label{curve3.fig}
\end{figure}

In Fig.~\ref{curve0.fig} the bifurcation diagram for the $L^2$ norm of $%
u_{-,k,0}$ is displayed for $\epsilon=0$ and $k\in (0,\frac34)$. Fig.~\ref%
{curve1.fig}--\ref{curve3.fig} illustrate the bifurcation diagrams for
different values of $\epsilon>0$, namely $0.1\sqrt{3}, \ 0.5\sqrt{3}$ and $%
0.9\sqrt{3}$. We plot the $L^2$ norm of the solutions against $k$. The lower
branches are obtained from $u_{-,k,\epsilon}$ with $k\in(\frac{\epsilon^2}{4}%
,\overline{k}_\epsilon)$, while the upper branches display the $L^2$ norm of
$u_{+,k,\epsilon}$ with $k\in(\frac34,\overline{k}_\epsilon)$. In each
diagram we observe the behavior predicted by the exact analysis of the
previous sections: the $L^2$ norm bifurcates from zero at $k=\frac{\epsilon^2}{4}$ and diverges
along the upper branch as $k\searrow \frac{3}{4}$, after `turning backwards'
at $k=\overline{k}_\epsilon$.

\begin{figure}[bh]
1) \includegraphics[width=60mm]{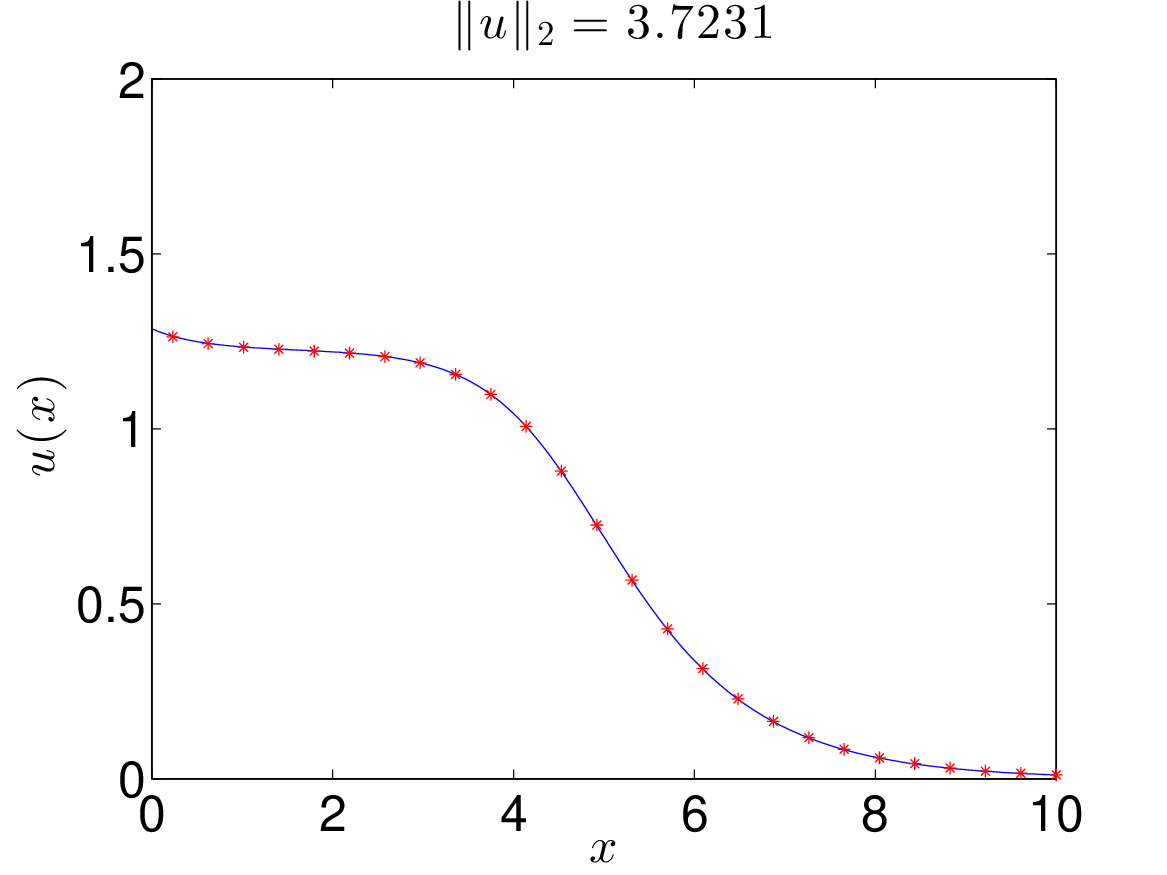} 2) %
\includegraphics[width=60mm]{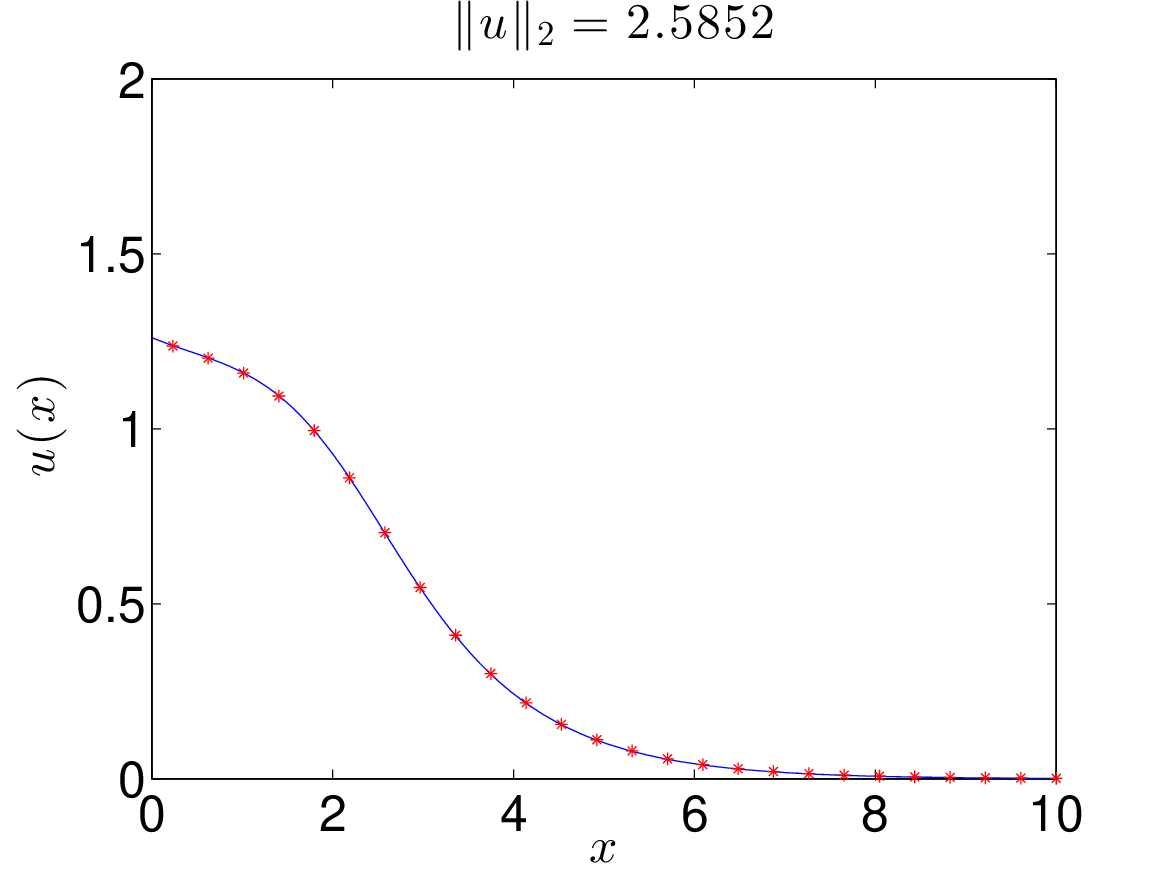} \newline
3) \includegraphics[width=60mm]{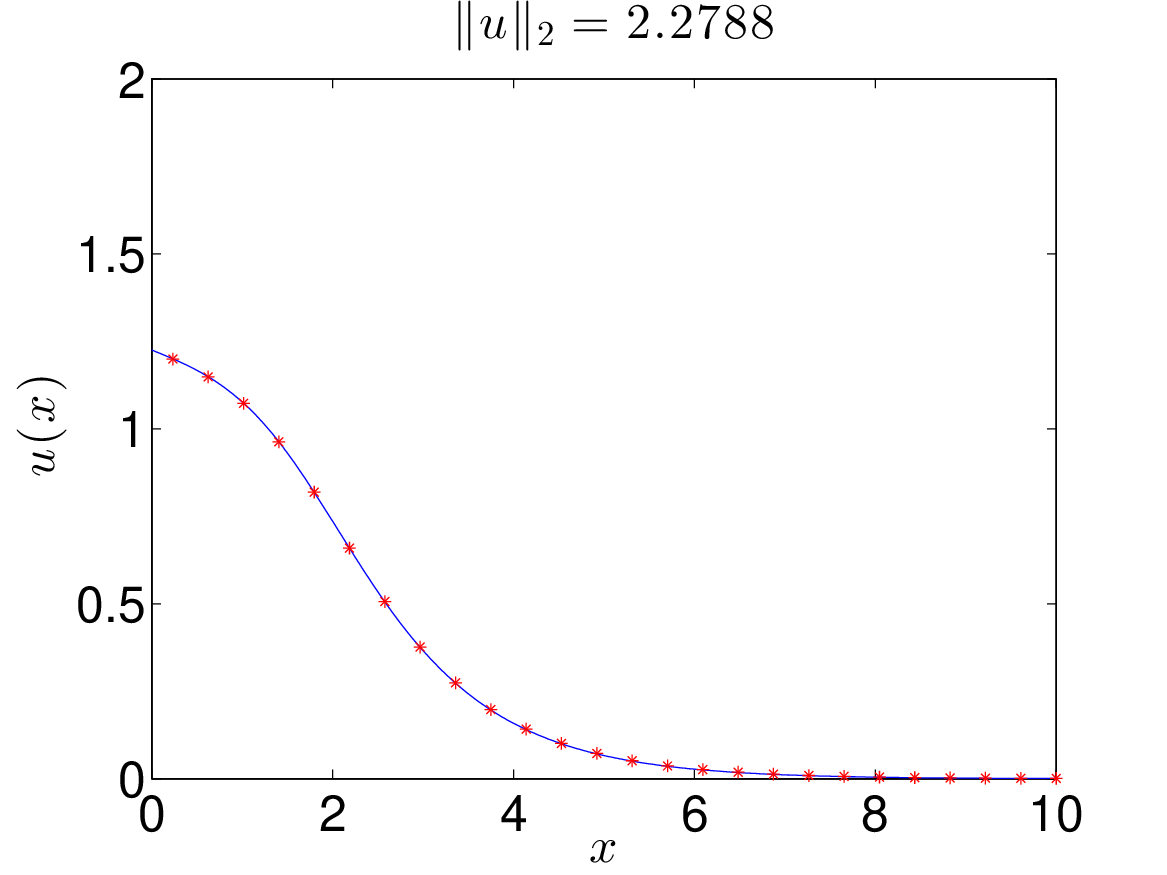} 4) %
\includegraphics[width=60mm]{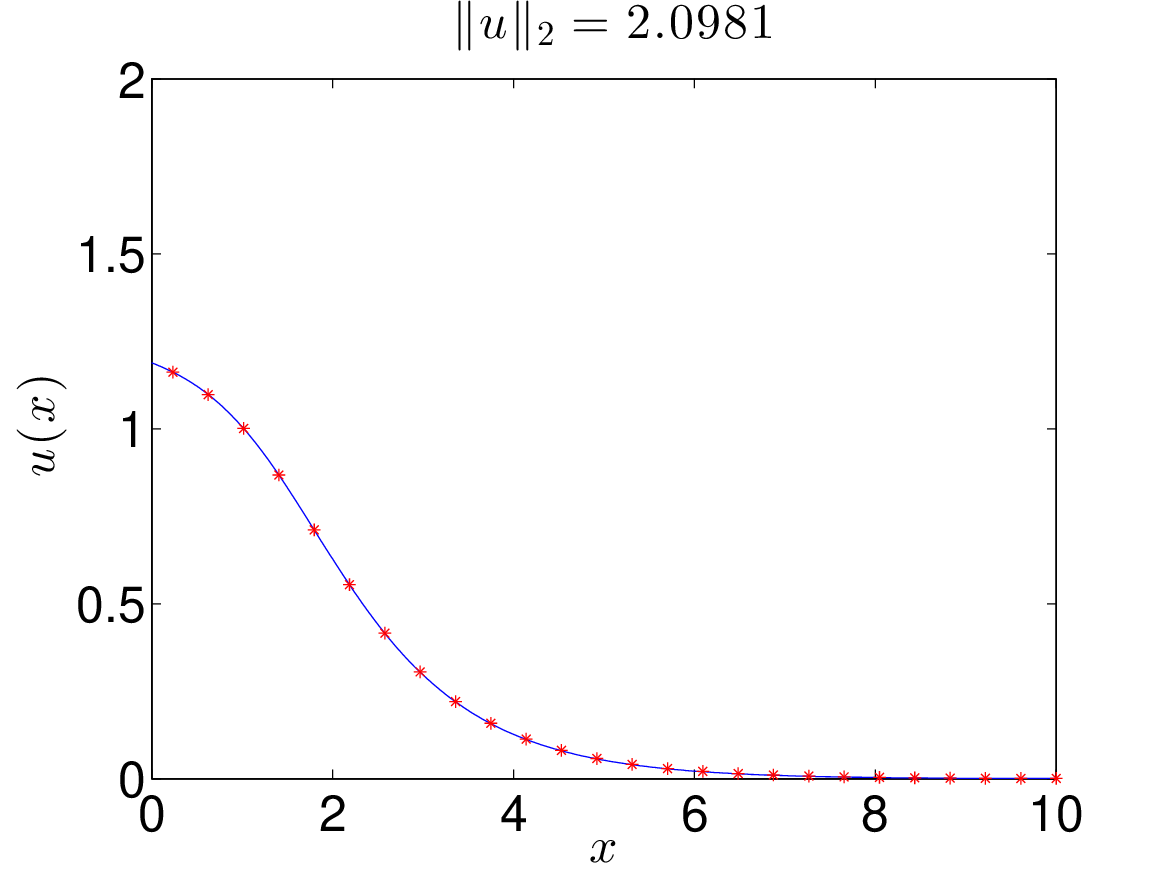} \newline
5) \includegraphics[width=60mm]{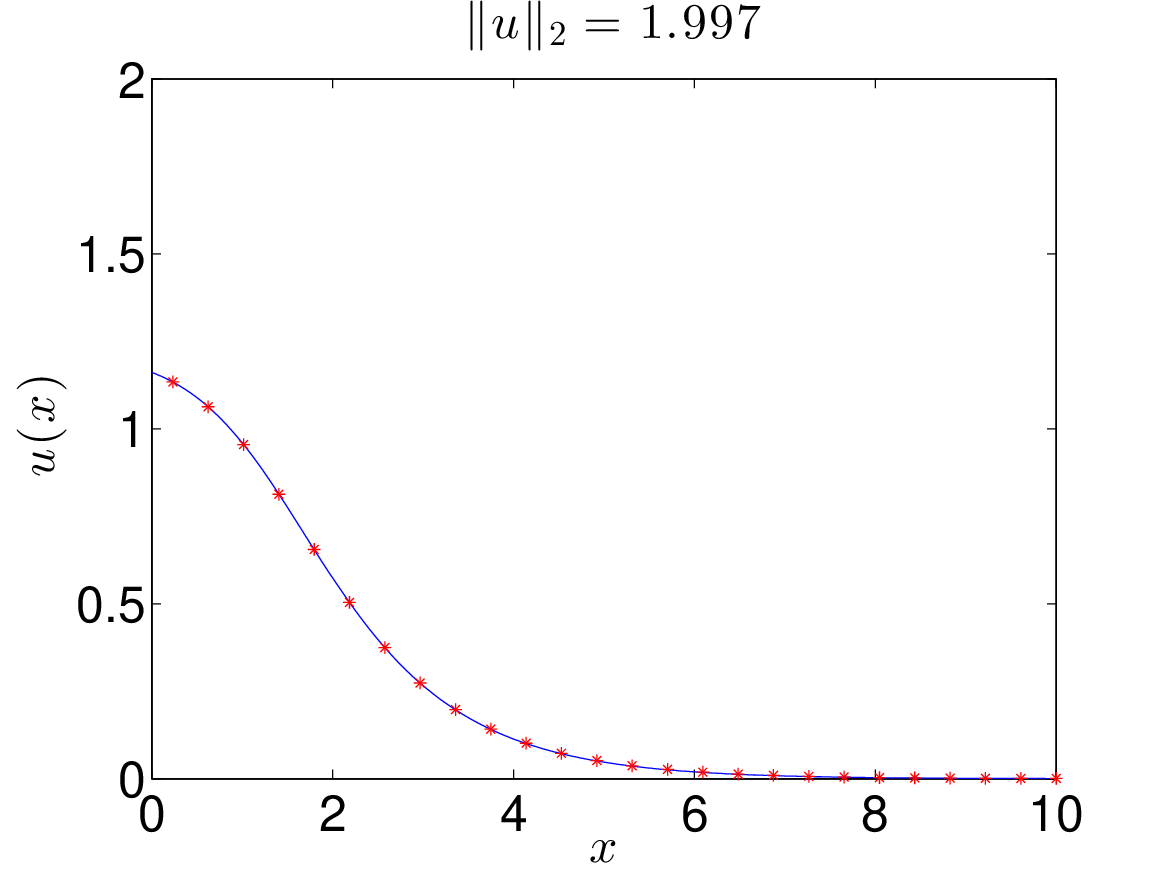} 6) %
\includegraphics[width=60mm]{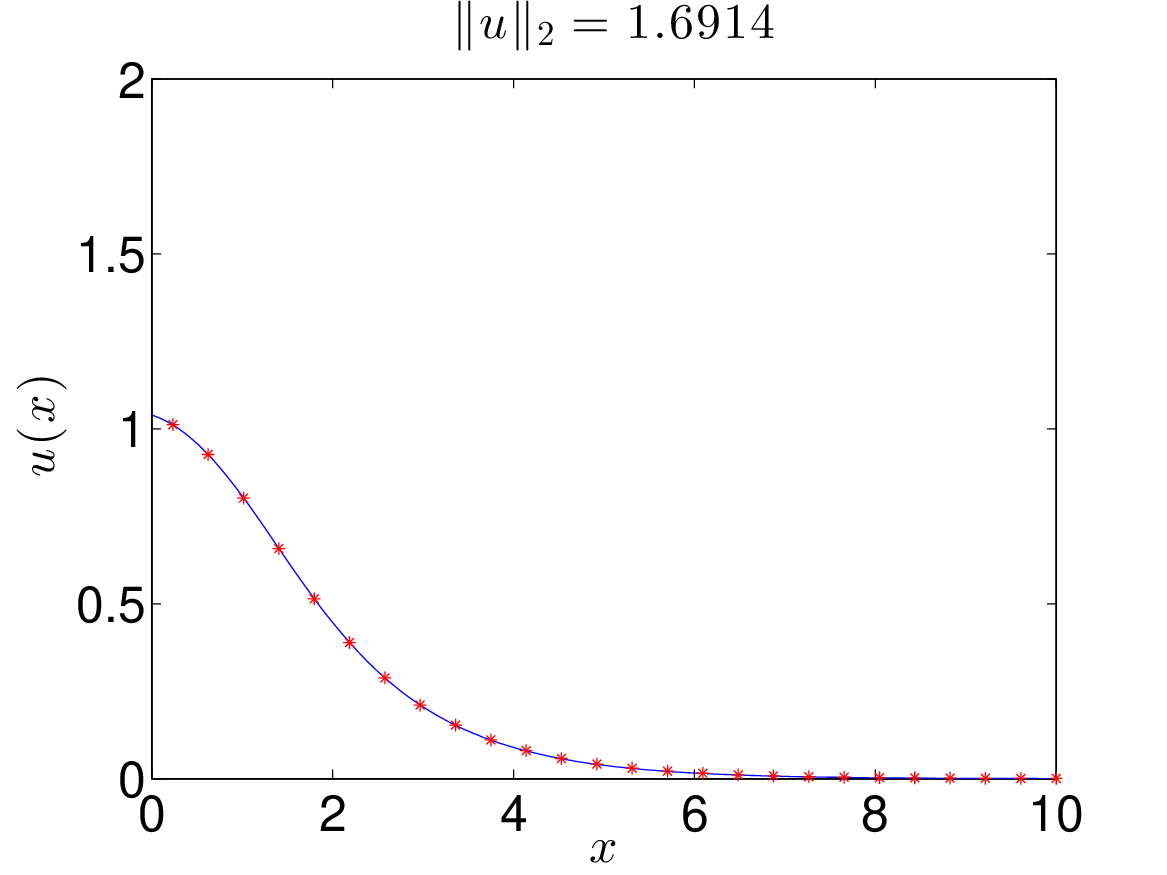} \newline
\caption{For $\protect\epsilon =0.1\protect\sqrt{3}$, in each of the plots
1)--6) we compare the discrete solution $u^{n}(x_{j})$ (\textcolor{red}{*}) to the
exact solution $u_{\pm ,k,\protect\epsilon }(x)$ (\textcolor{blue}{solid line}), corresponding
to the points on the bifurcation curve in Fig.~\protect\ref{curve1.fig}.}
\label{fig2}
\end{figure}

\begin{figure}[bh]
1) \includegraphics[width=60mm]{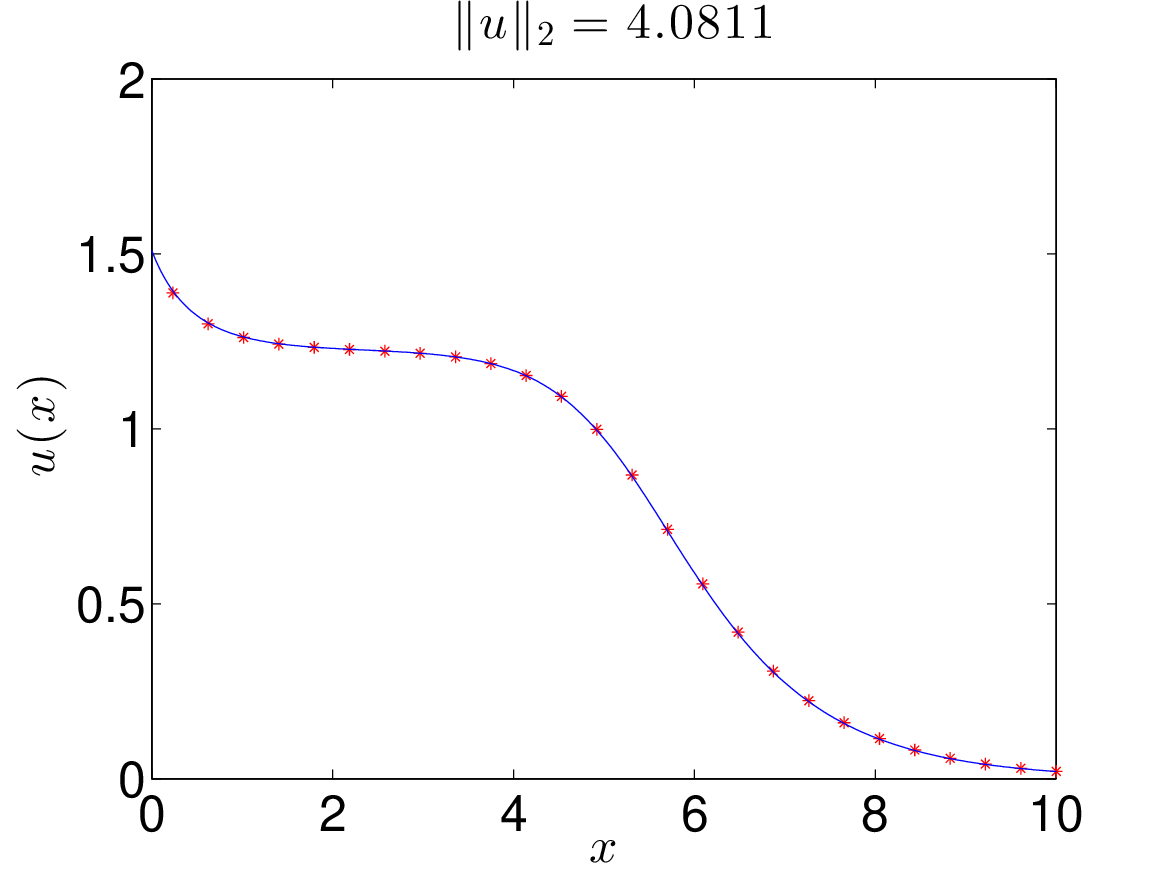} 2) %
\includegraphics[width=60mm]{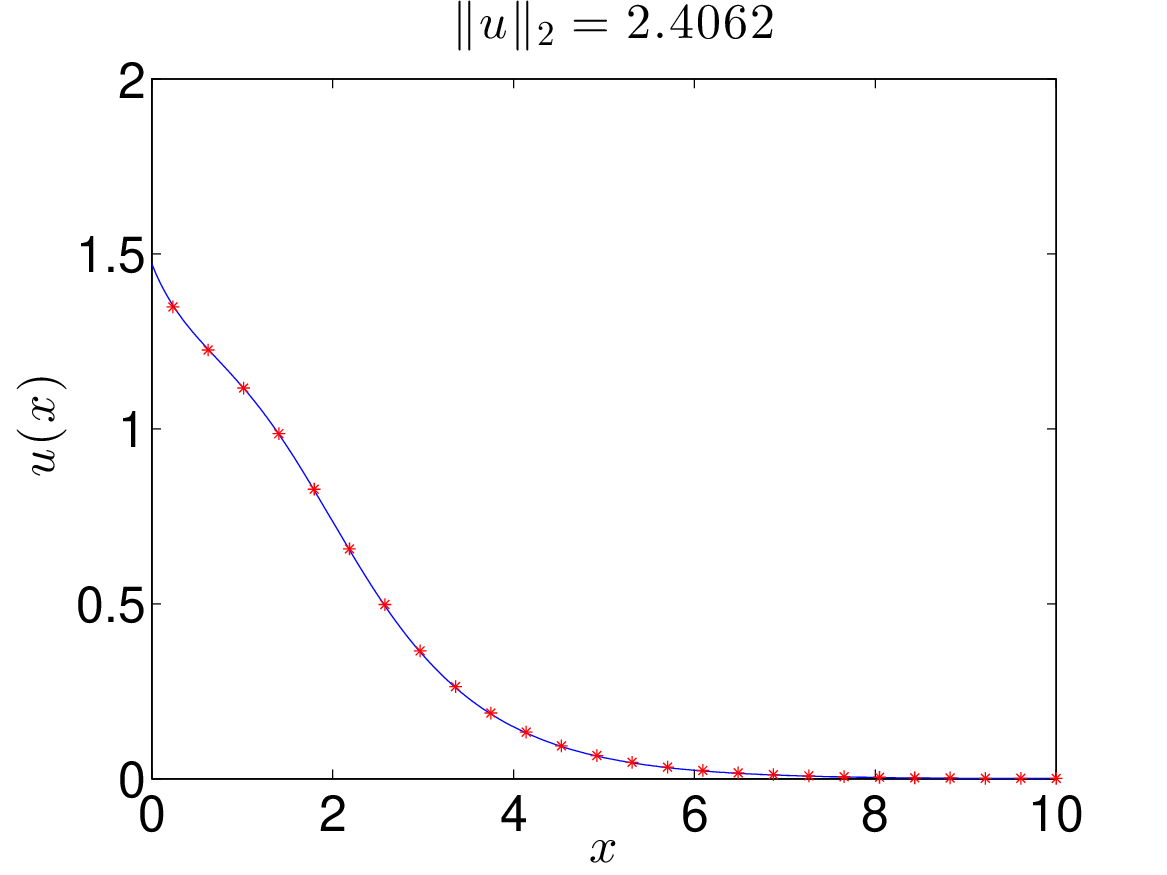} \newline
3) \includegraphics[width=60mm]{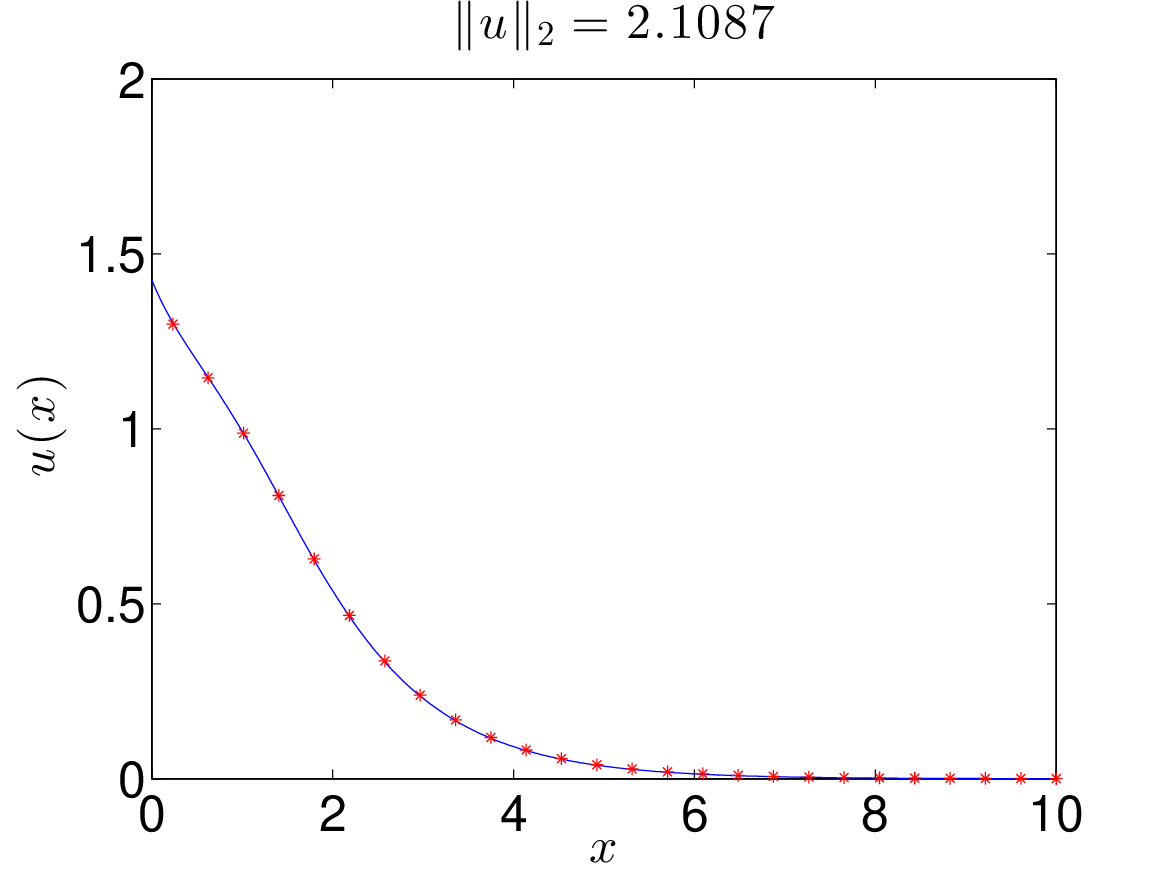} 4) %
\includegraphics[width=60mm]{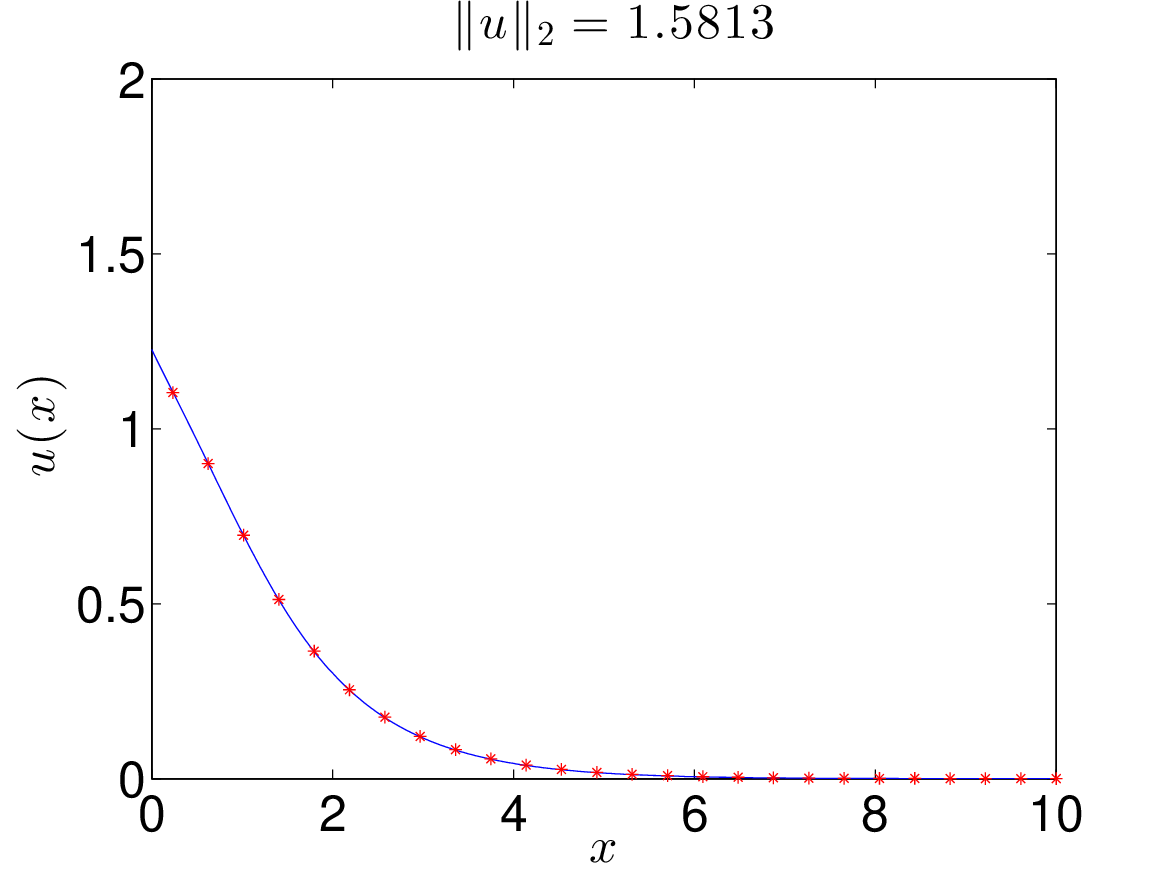} \newline
5) \includegraphics[width=60mm]{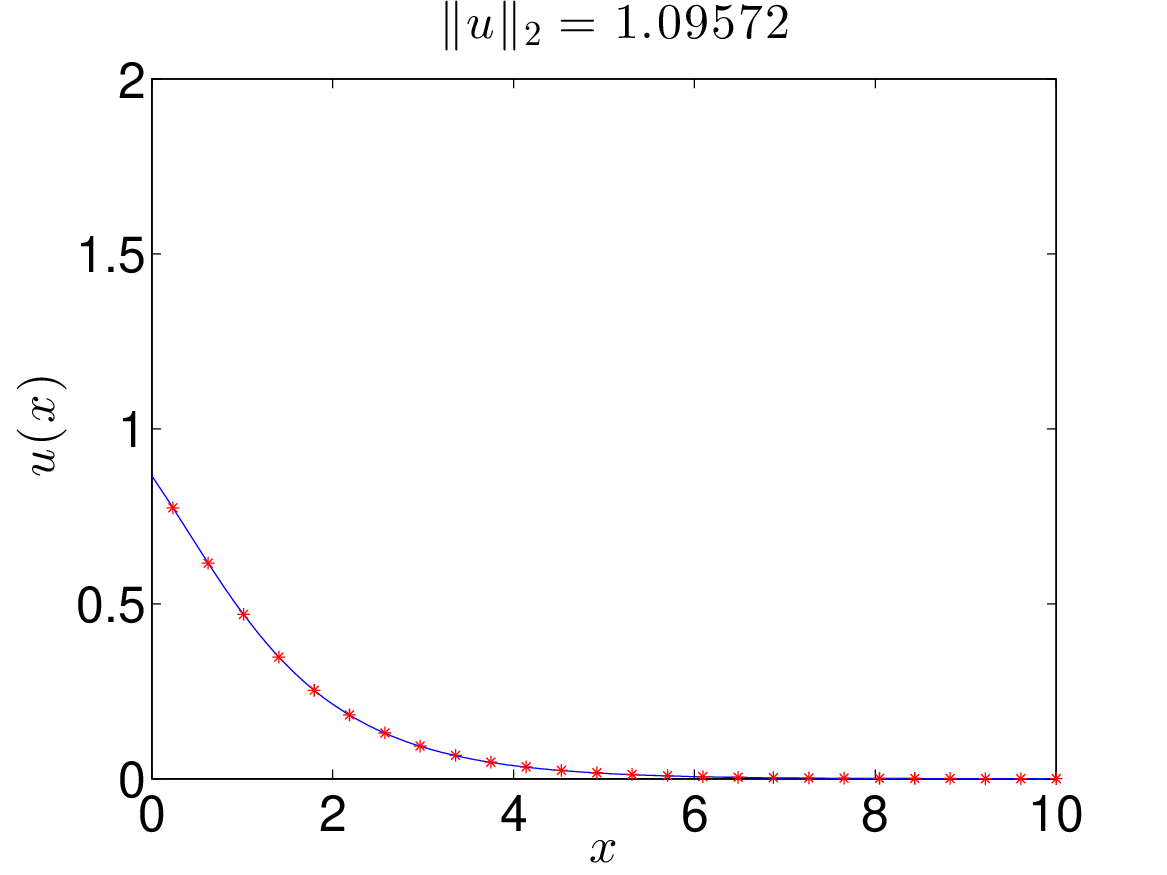} 6) %
\includegraphics[width=60mm]{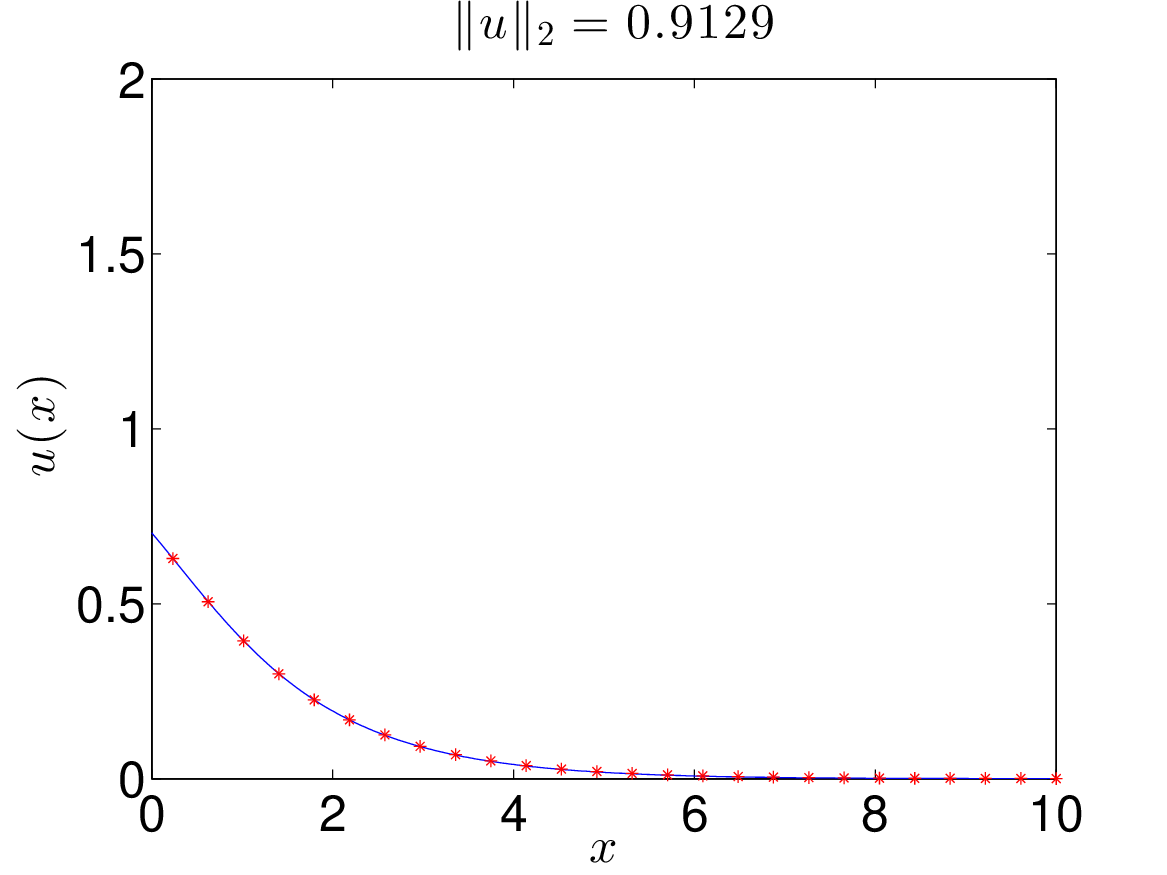} \newline
\caption{For $\protect\epsilon =0.5\protect\sqrt{3}$, in each of the plots
1)--6) we compare the discrete solution $u^{n}(x_{j})$ (\textcolor{red}{*}) to the
exact solution $u_{\pm ,k,\protect\epsilon }(x)$ (\textcolor{blue}{solid line}), 
corresponding to the points on the bifurcation curve in Fig.~\protect\ref{curve2.fig}.}
\label{fig4}
\end{figure}

\begin{figure}[bh]
1) \includegraphics[width=60mm]{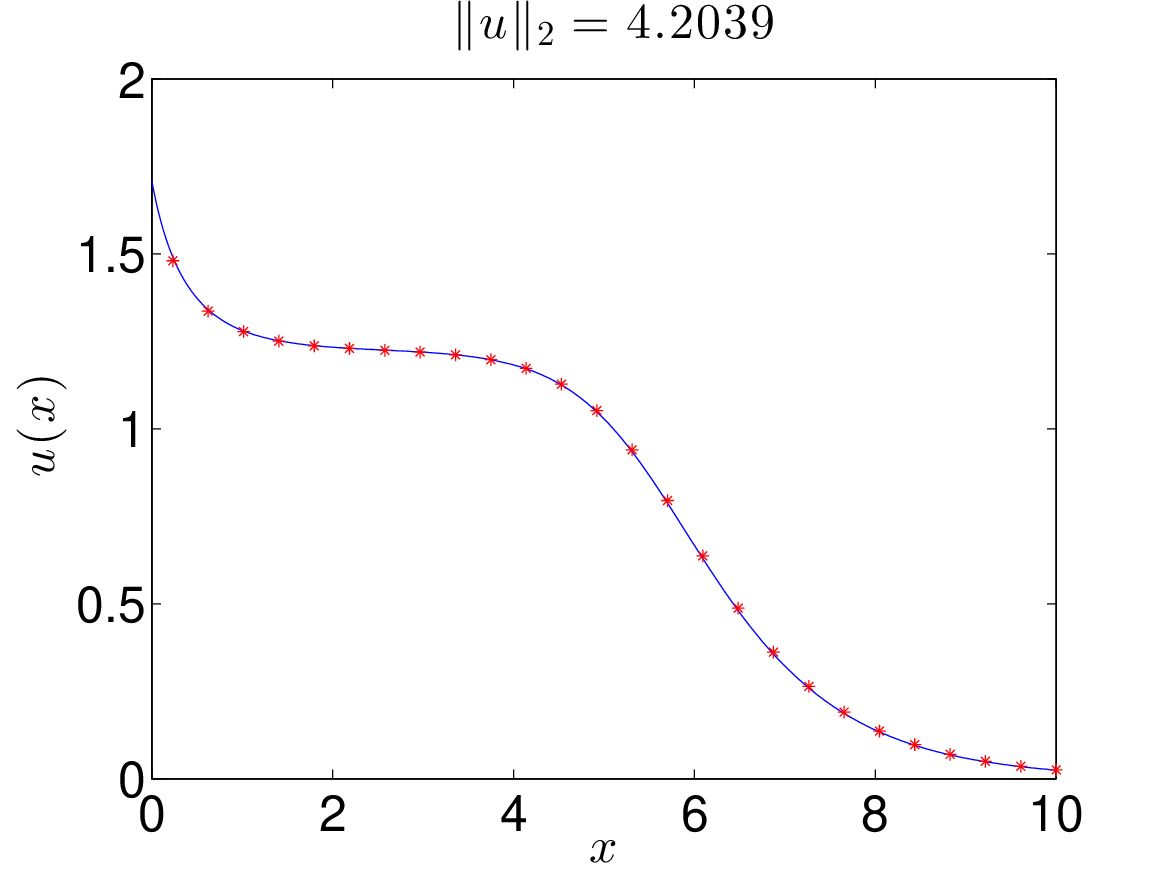} 2) %
\includegraphics[width=60mm]{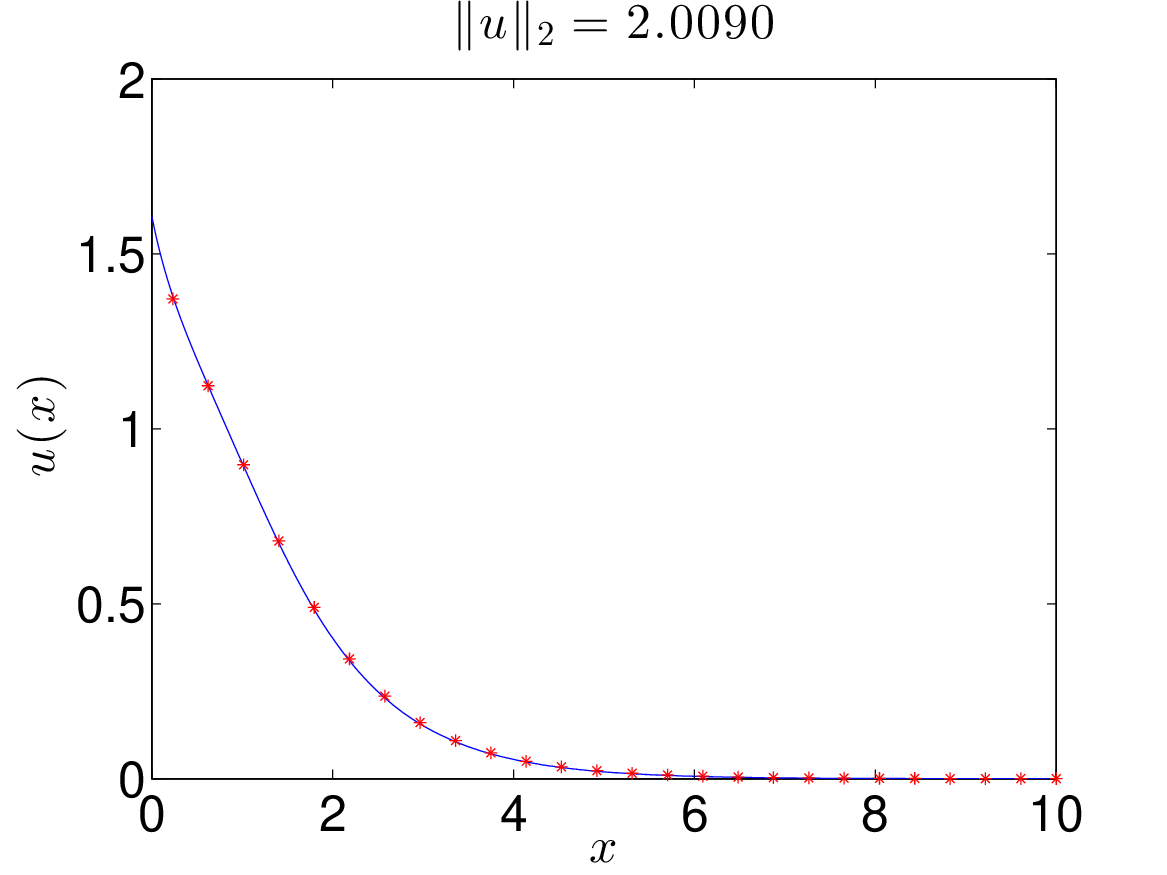} \newline
3) \includegraphics[width=60mm]{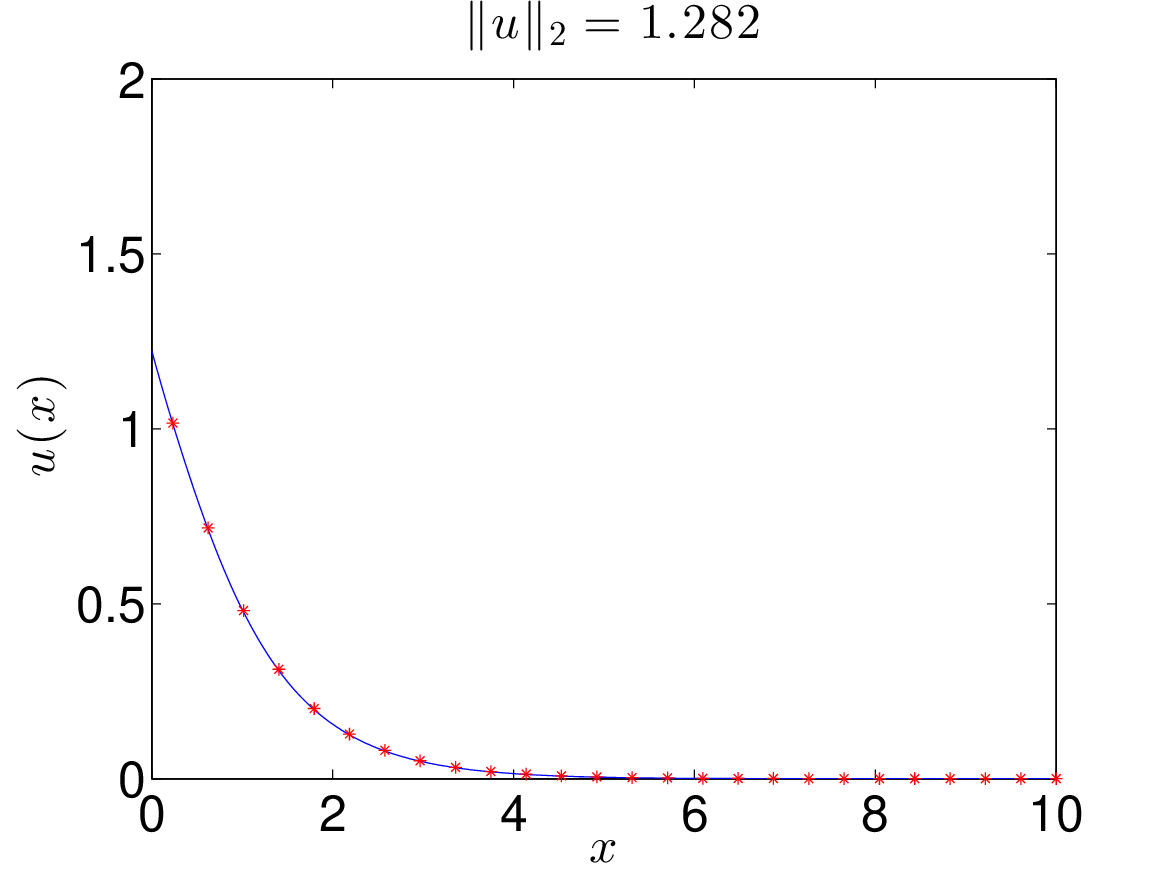} 4) %
\includegraphics[width=60mm]{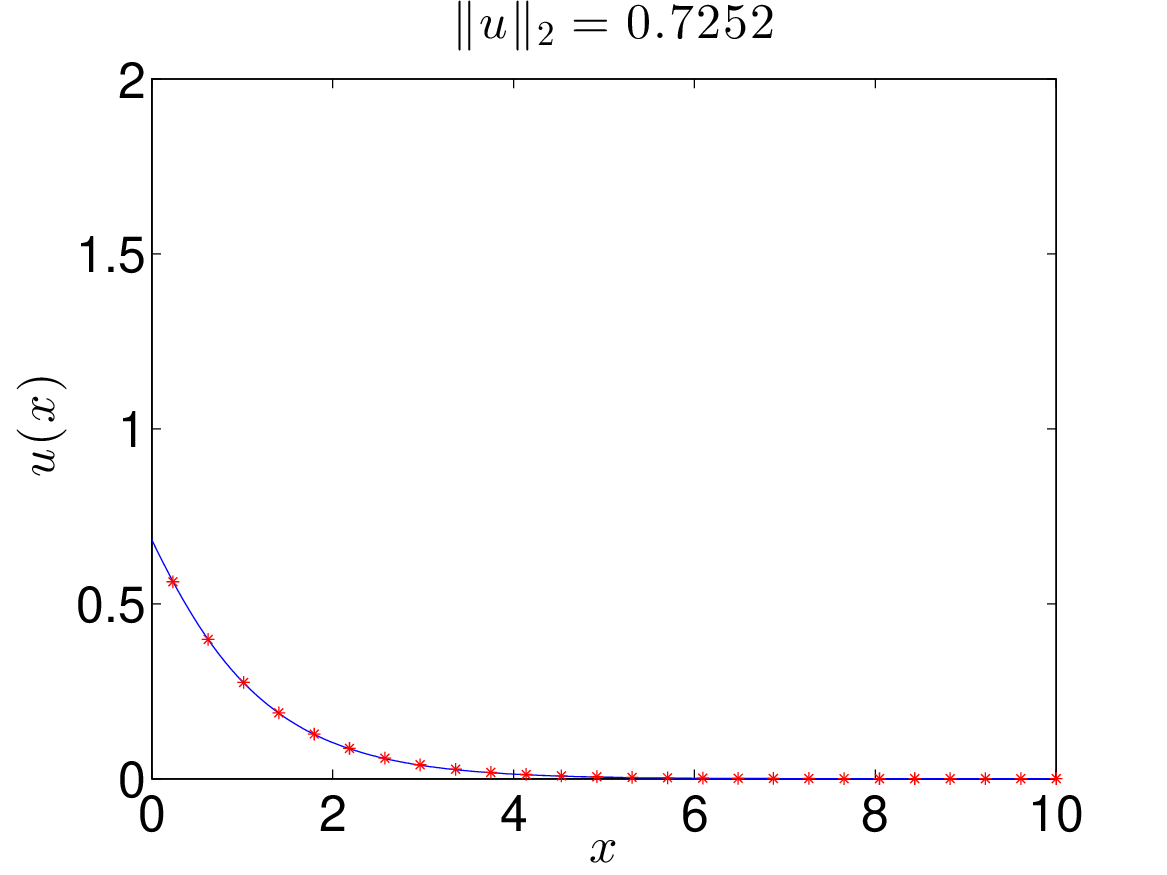} \newline
5) \includegraphics[width=60mm]{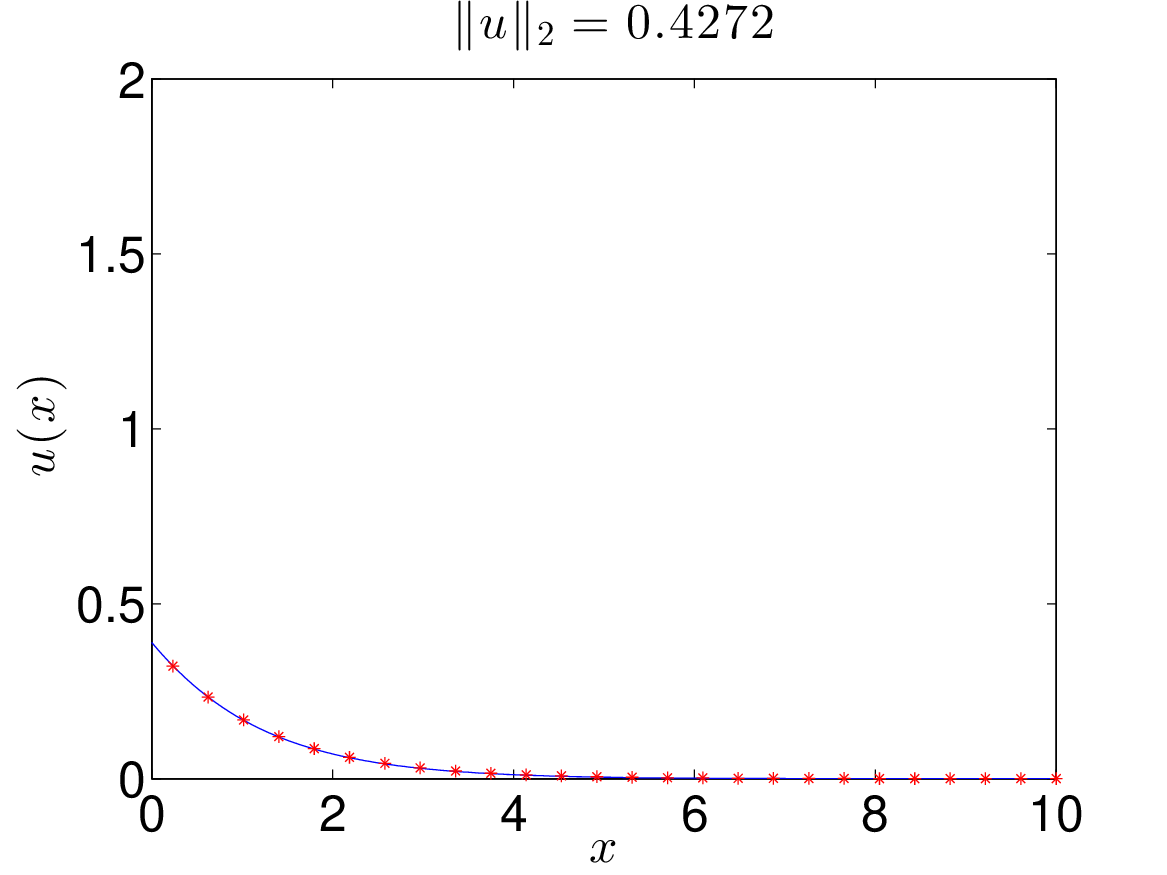} 6) %
\includegraphics[width=60mm]{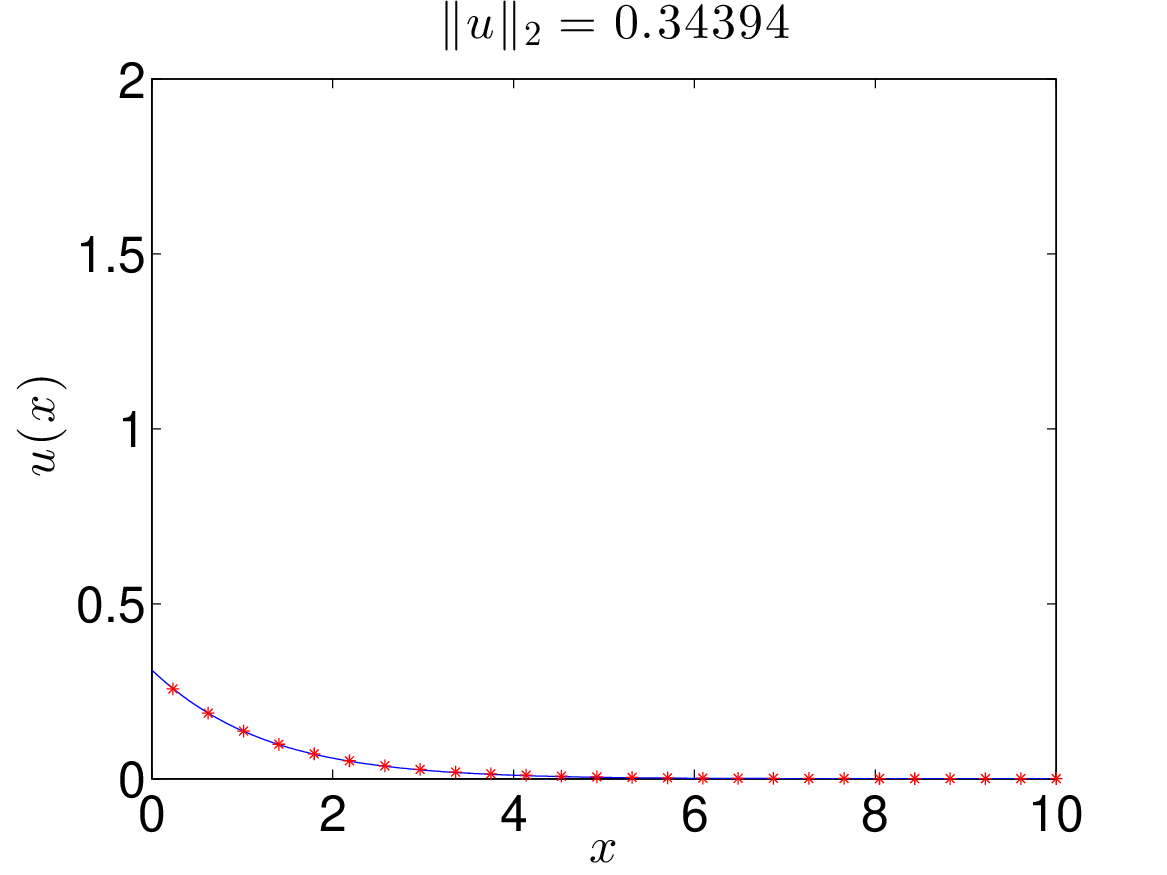} \newline
\caption{For $\protect\epsilon =0.9\protect\sqrt{3}$, in each of the plots
1)--6) we compare the discrete solution $u^{n}(x_{j})$ (\textcolor{red}{*})
to the exact solution $u_{\pm ,k,\protect\epsilon }(x)$ (\textcolor{blue}{solid line}), 
corresponding to the points on the
bifurcation curve in Fig.~\protect\ref{curve3.fig}.}
\label{fig6}
\end{figure}

In Fig.~\ref{curve1.fig}, for $\epsilon=0.1\sqrt{3}$, we pick up six
different values $a_1,\dots,a_6$ of the $L^{2}$ norm
for which we then compare, in Fig.~\ref{fig2}, the exact solutions $%
u_{\pm,k,\epsilon}$ to the numerical solutions $u^n(x_j)$ obtained by
minimization under the constraints $a_1$ to $a_6$; see plots $1)-6)$ in
Fig.~\ref{fig2}. In contrast to the analytical computation of the solution,
where we had to distinguish the two branches $u_{\pm,k,\epsilon}$ in the
range $k\in(\frac34,\overline{k}_\epsilon)$, the numerical computation is performed 
at fixed $L^2$ norm, and provides the only solution corresponding to each
given value $a=\Vert u\Vert _{L^{2}}$.

Similarly, in Fig.~\ref{fig4} and Fig.~\ref{fig6} we compare the exact
solutions with the numerical ones, for $\epsilon =0.5\sqrt{3}$ and $\epsilon
=0.9\sqrt{3}$, respectively. We again use six values of the $L^{2}$ norm
taken from the corresponding bifurcation diagrams in Fig.~\ref{curve2.fig}
and Fig.~\ref{curve3.fig}. In all three cases we notice an excellent
agreement between the exact solutions and the numerical ones. So far, the
continuous normalized gradient flow (CNGF) has been mostly used for the NLS
with a cubic nonlinearity. Our results demonstrate its effectiveness in the
case of a cubic-quintic nonlinearity. Furthermore, the CNGF being
variational in nature, this suggests that the positive solutions of %
\eqref{solequ} should admit a variational characterization --- as obtained for
instance in \cite{jf} in the case of a single power nonlinearity. This will
be discussed further elsewhere.



\end{document}